\documentclass[a4]{amsart}

\input xypic 
\input xy 
\xyoption{all} 
\usepackage{amssymb} 
\usepackage{bbm}
\usepackage{tikz-cd}

\usepackage{float}
\usepackage{graphicx}
\usepackage{accents}

\usepackage[bookmarksnumbered, bookmarksopen, backref]{hyperref}

\newtheorem{theorem}{Theorem}[section]
\newtheorem{proposition}[theorem]{Proposition}
\newtheorem{lemma}[theorem]{Lemma}
\newtheorem{corollary}[theorem]{Corollary}

\theoremstyle{definition}
\newtheorem{definition}[theorem]{Definition}
\newtheorem{remark}[theorem]{Remark}
\newtheorem{example}[theorem]{Example}

\newcounter{RomanNumber}
\newcommand{\MyRoman}[1]{\setcounter{RomanNumber}{#1}\Roman{RomanNumber}}
 
\newcommand{\conn}{\ensuremath{\#}}

\newcounter{bean}

\newcommand{\namedright}[3]{\ensuremath{#1\stackrel{#2}
 {\longrightarrow}#3}}
\newcommand{\nameddright}[5]{\ensuremath{#1\stackrel{#2}
 {\longrightarrow}#3\stackrel{#4}{\longrightarrow}#5}}
\newcommand{\namedddright}[7]{\ensuremath{#1\stackrel{#2}
 {\longrightarrow}#3\stackrel{#4}{\longrightarrow}#5
  \stackrel{#6}{\longrightarrow}#7}}

\newcommand{\larrow}{\relbar\!\!\relbar\!\!\rightarrow}
\newcommand{\llarrow}{\relbar\!\!\relbar\!\!\larrow}

\newcommand{\lnameddright}[5]{\ensuremath{#1\stackrel{#2}
 {\larrow}#3\stackrel{#4}{\larrow}#5}}

\newcommand{\llnameddright}[5]{\ensuremath{#1\stackrel{#2}
 {\llarrow}#3\stackrel{#4}{\llarrow}#5}}

\newcommand{\qqed}{\hfill\Box}

\begin{document}
\begin{sloppypar}

\title{Homotopy of blow ups after looping} 

\author{Ruizhi Huang} 
\address{Institute of Mathematics, Academy of Mathematics and Systems Science, 
 Chinese Academy of Sciences, Beijing 100190, China} 
\email{huangrz@amss.ac.cn} 
   \urladdr{https://sites.google.com/site/hrzsea/}
   \thanks{}
   
\author{Stephen Theriault}
\address{School of Mathematics, University of Southampton, Southampton 
   SO17 1BJ, United Kingdom}
\email{S.D.Theriault@soton.ac.uk}

\subjclass[2010]{Primary 
55P35, 
57R22, 
57R19,  
14F35;  
Secondary 
55Q52, 
57R65,  
55Q50,  
55P40,  
55P62. 
}
\keywords{blow up, surgery, fibrewise homotopy, loop space, gauge group}
\date{}


\begin{abstract} 
The homotopy theory of the blow up construction in algebraic and symplectic geometry is 
investigated via two approaches. The first approach introduces and develops fibrewise 
surgery theory, for which the fibrewise framing is characterized by the homotopy groups 
of a certain gauge group. This is used to obtain a homotopy decomposition of the based loop 
space on a blow up that holds $p$-locally for all but finitely many primes $p$ and holds rationally. 
The second approach is purely homotopy theoretic and obtains a homotopy decomposition of the based 
loop space on a blow up that holds integrally provided a certain condition is satisfied by an associated 
homotopy action. 

As applications, we obtain $p$-local homotopy decompositions of the based loop 
space of a focal genus $2$ manifold, improve an earlier result of the authors on the homotopy of 
manifolds stabilized by projective spaces, and obtain for blow-ups a refinement of 
the rational dichotomy.
\end{abstract}

\maketitle
\tableofcontents

\section{Introduction}
The blow up construction is fundamental to both complex algebraic geometry and symplectic geometry (Gromov \cite{Gro} and McDuff \cite{Mc}). 
Via Sullivan's rational homotopy theory \cite{Sul}, it is particularly useful for constructing non-K\"{a}hler symplectic manifolds \cite{Thu, BaT, RT, TO}. While the rational homotopy type of blow ups has been systematically studied in depth by Lambrechts and Stanley \cite{LS1, LS2, LS3}, the local homotopy type of blow ups has received little attention.

In this paper, we study the local homotopy of blow ups by mixing techniques from geometric topology and unstable homotopy theory. 
To explain our main results, let us fix some initial data. 
Following \cite{LS3}, we work in a general context which is sufficient for defining the blow up construction.
Let $B\hookrightarrow N$ be a codimension $2n$ embedding of connected closed oriented smooth manifolds. 
Suppose the normal bundle $\nu$ of the embedding has been given a complex structure. 
As in \cite{Mc}, \cite[Section~2]{LS3}, or Section \ref{sec: blowup} below, we can define the {\it blow up} 
$\widetilde{N}$ of $N$ along $B$ as follows. Let $V$ be a closed tubular neighborhood of $B$ in $N$ 
diffeomorphic to the disk bundle $D\nu$ of $\nu$; $V$ and this disc bundle shall usually not be distinguished. 
Let $N_c=\overline{N\backslash V}$ be the closure of the complement of $V$ in $N$, and notice that 
$\partial N_c=\partial V$. The manifold $N$ is the pushout of the inclusions 
\(\namedright{\partial V}{}{V}\) 
and 
\(\namedright{\partial V}{}{N_{c}}\). 
Replacing the inclusion 
\(\namedright{\partial V}{}{V\cong D\nu}\) 
by the projectivization 
\(\namedright{\partial V}{}{P\nu}\), 
a topological interpretation of the blow up $\widetilde{N}$ is as the pushout of 
\(\namedright{\partial V}{}{P\nu}\) 
and 
\(\namedright{\partial V}{}{N_{c}}\).
\medskip 

\noindent 
\textbf{Results via surgery in Part \ref{part1}}.  

Our first result concerns the local homotopy of blow ups in the stable range. This is done by introducing fibrewise surgery theory and mixing it with Cube Lemma techniques in homotopy theory. We state the result first and discuss the methodology after seeing applications and examples.  
To ensure the existence of localization, we assume that the manifolds involved $N$, $B$, $N_c$, $\partial V$ and $\widetilde{N}$ are all connected nilpotent spaces. In particular, this is the case when~$N$ and $B$ are simply connected and $n\geq 2$ as is shown in Lemma \ref{nillemma}. 
For any $CW$-complex~$X$, let~$\Sigma^i X$ be the $i$-fold suspension of~$X$ and let $\Omega X$ be its based loop space. 
\begin{theorem}\label{loopblowupthmintro}
Let $N$ be a connected closed oriented smooth manifold. 
Let $\widetilde{N}$ be the blow up of $N$ along a $(k-1)$-connected $l$-dimensional submanifold $B$ with a normal bundle of complex dimension~$n$.
Suppose that $l\leq 2n-4$ and $k\geq 1$. 
Let $p$ be a prime such that $p>\frac{1}{2}(l-k)+1$, $(p-1) \nmid 2s$ for any $k+2\leq 4s \leq l+2$, and $H_\ast(B;\mathbb{Z})$ is $p$-torsion free. 
Then there is a $p$-local homotopy equivalence 
\[
\Omega \widetilde{N}\simeq S^1\times \Omega N_c\times \Omega \Sigma^2 F,
\]
where $F$ is the homotopy fibre of the inclusion $ \partial V\hookrightarrow N_c$ of the boundary.
\end{theorem}

The local condition in Theorem \ref{loopblowupthmintro} is due to the role of a gauge group in the fibrewise framing of a fibrewise surgery and will be explained in the lead up to Theorem~\ref{method_theorem}. 

Theorem \ref{loopblowupthmintro} describes the loop space homotopy type of 
a blow up in terms of its internal structure. 
In particular, it allows us to compute the homotopy groups of~the blow up $\widetilde{N}$ in terms of those of~$S^1$, the complement $N_c$ and the double suspension of the homotopy fibre $F$. 
Since a blow up along a point is a connected sum with a complex projective space, Theorem \ref{loopblowupthmintro} generalizes an earlier result of the authors \cite{HT2} about the homotopy type of 
$\Omega(N\conn\mathbb{C}P^{n})$. 

From an embedding point of view, Theorem \ref{loopblowupthmintro} implies that the homotopy theory of the blow up~$\widetilde{N}$ is strongly related to the homotopy theory of the complement $N_c$ of the embedding. 
This point has two consequences. 
First, a subtle fact is that different embeddings may lead to very different complements. For instance, in knot theory, the famous Gordon-Luecke \cite{GL} states that a knot in~$S^3$ is determined by its complement, so the complement of a circle embedded in $S^{3}$ can vary considerably. In our case, different embeddings of $B$ in $N$ may lead to different complements $N_c$, and hence different blow-ups $\widetilde{N}$ with very different homotopy theory. 
Second, there is no essential relation between the homotopy theory of a manifold and the homotopy theory of its submanifold complement. This again can be easily seen from knot theory. 
However, as the homotopy theory of the blow up~$\widetilde{N}$ is strongly related to the homotopy theory of the complement $N_c$, in general it is dramatically different from the homotopy theory of $N$ (for example, it would not be expected that the loops on $N_{c}$ retracts off the loops on $N$).

Nevertheless, in the following example we provide a family of manifolds $N$ such that the complement $N_c$ and the homotopy fibre $F$ in Theorem~\ref{loopblowupthmintro} can be well understood. 

\begin{example}[Focal genus $2$ manifolds]\label{genus2ex}
Let $N$ be a closed manifold such that it can be decomposed into a union of two disk bundles
\[
N\cong D\xi\cup_g D\nu,
\]
where $\xi$ and $\nu$ are oriented real vector bundles over manifolds $A$ and $B$ respectively, and $g: S\xi\stackrel{}{\rightarrow} S\nu$ is a diffeomorphism between their sphere bundles.
Following \cite{D2}, $N$ is called a manifold of {\it focal genus $2$}. In differential geometry, interesting examples of focal genus $2$ manifolds appear in the study of Dupin hypersurfaces \cite{Tho}, including isoparametric hypersurfaces \cite{Mun, Heb}, and the principal orbits of strict cohomogeneity one actions \cite{Mos}. 
In algebraic topology, a focal genus $2$ manifold is a special case of a double mapping cylinder, the rational homotopy of which was thoroughly studied by Grove and Halperin \cite{GH}.

Suppose $\xi$ is of dimension $m$, $\nu$ is of dimension $2n$, and $\nu$ admits a complex structure. Then there is a blow up $\widetilde{N}$ of $N$ along the zero section $B$ of $D\nu$. 
Following the notation and hypotheses in Theorem \ref{loopblowupthmintro}, it is clear that $V\cong D\nu\simeq B$, $N_c\cong D\xi\simeq A$ and $\partial V\cong S\xi\cong S\nu$. Since the inclusion $S\xi \stackrel{}{\hookrightarrow}D\xi$ of the boundary is homotopic to the sphere bundle projection $S\xi \stackrel{}{\rightarrow} A$, the homotopy fibre~$F$ is homotopy equivalent to $S^{m-1}$.
Then under the hypotheses of Theorem \ref{loopblowupthmintro} there is a $p$-local homotopy equivalence
\begin{equation} 
  \label{focalgenusdecomp} 
  \Omega \widetilde{N}\simeq S^1\times \Omega A\times \Omega S^{m+1}. 
\end{equation} 
\end{example}

In rational homotopy theory, Lambrechts and Stanley \cite[Theorem 7.8]{LS3} determined the rational homotopy type of the blow up $\widetilde{N}$ in the stable range in terms of the rational homotopy type of the embedding $B\hookrightarrow N$ and its rational Chern classes. In local homotopy theory, Theorem \ref{loopblowupthmintro} implies that at any sufficiently large prime $p$ the loop space homotopy type of~$\widetilde{N}$ only depends on the complement of the embedding, in terms of $N_{c}$ itself and the homotopy fibre of the inclusion $\partial N_{c}=\partial V\hookrightarrow N_c$. 
Furthermore, rationally the statement of Theorem~\ref{loopblowupthmintro} simplifies and holds 
in greater generality.

\begin{theorem}\label{loopblowupthmintrorational}
Let $N$ be a connected closed oriented smooth manifold.  
Let $\widetilde{N}$ be the blow up of~$N$ along a connected $l$-dimensional submanifold $B$ with a normal bundle of complex dimension $n$. Suppose that $l\leq 2n-4$. Then there is a rational homotopy equivalence 
\[
\Omega \widetilde{N}\simeq S^1\times \Omega N_c\times \Omega \Sigma^2 F,
\]
where $F$ is the homotopy fibre of the inclusion $ \partial V\hookrightarrow N_c$ of the boundary.
\end{theorem}

An interesting application of Theorem \ref{loopblowupthmintrorational} is a refinement of the classical dichotomy characterizing rational spaces~\cite[page 452]{FHT}, \cite[Section 2.5.3]{FOT}. The dichotomy states: 

\textit{Any connected nilpotent space $X$ with rational homology of finite type and finite rational Lusternik-Schnirelmann category is either:
\begin{itemize}
\item rationally elliptic, that is, $\pi_\ast(X)\otimes \mathbb{Q}$ is finite dimensional, or else
\item rationally hyperbolic, that is, $\pi_\ast(X)\otimes \mathbb{Q}$ grows exponentially.
\end{itemize}
}
\noindent 
A connected finite dimensional $CW$-complex, for example, has finite Lusternik-Schnirelmann 
category. Since a connected compact smooth manifold can be given the structure of a finite dimensional $CW$-complex, it has rational homology of finite type and finite rational Lusternik-Schnirelmann category.

Let $\mathbb{Q}(i)$ be a copy of the vector space $\mathbb{Q}$ of degree $i$.

\begin{theorem}
\label{dichotomythmintro} 
Let $N$ be a connected closed oriented smooth manifold.  
Let $\widetilde{N}$ be the blow up of $N$ along an $l$-dimensional submanifold $B$ with a normal bundle of complex dimension $n$. 
Suppose that $l\leq 2n-4$. 
Then either:
\begin{itemize}
\item[(1)] $\widetilde{N}$ is rationally elliptic, in which case $N_c$ is also rationally elliptic, the homotopy fibre of 
the inclusion $\partial V\hookrightarrow N_c$ is rationally homotopy equivalent to a sphere $S^{q}$, and 
\[
\pi_\ast(\widetilde{N})\otimes \mathbb{Q}\cong (\pi_\ast(N_c)\otimes \mathbb{Q})\oplus
\left\{\begin{array}{cc}
\mathbb{Q}(2)\oplus \mathbb{Q}(q+2)& q \ {\rm is} \ {\rm odd}, \\
\mathbb{Q}(2)\oplus \mathbb{Q}(q+2)\oplus \mathbb{Q}(2q+3) & q \ {\rm is} \ {\rm even};
\end{array}\right. 
\]
\item[(2)] $\widetilde{N}$ is rationally hyperbolic, in which case either $N_c$ is rationally hyperbolic or the homotopy fibre of the inclusion $\partial V\hookrightarrow N_c$ is not rationally homotopy equivalent to a sphere.
\end{itemize}
\end{theorem}

In practice, it is usually much harder to construct rationally elliptic manifolds than rationally hyperbolic manifolds. For instance, a result of Beben and the second author \cite{BT1} implies that most $(n-1)$-connected closed $2n$-manifolds are rationally hyperbolic. Nevertheless, a particular aspect of the following example indicates that we can construct examples of rationally elliptic manifolds from the blow up construction of rationally elliptic focal genus $2$ manifolds.

\begin{example}[Focal genus $2$ manifolds continued]\label{genus2ex.2}
Let $N$ be the focal genus $2$ manifold in Example~\ref{genus2ex}. By either the local homotopy equivalence~(\ref{focalgenusdecomp}) or Theorem \ref{dichotomythmintro}, the blow up $\widetilde{N}$ is rationally elliptic if and only if $A$ is rationally elliptic. On the other hand, a result of Grove and Halperin \cite{GH} implies that $A$ is rationally elliptic if and only if $N$ is.
Hence, for a focal genus $2$ manifold, the blow up construction preserves rational ellipticity and rational hyperbolicity. 
\end{example}

Theorems \ref{loopblowupthmintro}, \ref{loopblowupthmintrorational} and~\ref{dichotomythmintro} will be proved in Part \ref{part1} from a surgery theory point of view. To motivate the ideas involved, we first briefly review the strategy in \cite{HT2} for studying the loop space homotopy type of $N\#\mathbb{C}P^n$, the connected sum of $N$ and the complex projective space $\mathbb{C}P^n$. This connected sum can be viewed as the effect of a $0$-surgery on $N\amalg \mathbb{C}P^n$. Moreover, there is a canonical circle bundle over $N\#\mathbb{C}P^n$ determined by a generator of $H^2(\mathbb{C}P^n)$, such that its total manifold $E$ is the effect of a $1$-surgery on $N\times S^1$, the product manifold of $N$ with the circle $S^1$. This leads to a homotopy equivalence $\Omega(N\#\mathbb{C}P^{n})\simeq S^{1}\times\Omega E$, so to study the homotopy type of $\Omega(N\#\mathbb{C}P^n)$ it is equivalent to study the homotopy type of $\Omega E$.
To do this we consider the framing of the $1$-surgery, which is determined by a map from $S^1$ to the special orthogonal group $SO(2n)$. Recall the classical $J$-homomorphism can be realized by a map $SO(2n) \stackrel{J}{\longrightarrow} \Omega^{2n}S^{2n}$. 
It turns out that the homotopy effect of the framing is determined by a class in the image of $\pi_1(SO(2n)) \stackrel{J_\ast}{\longrightarrow} \pi_1(\Omega^{2n}S^{2n})$ that is of order at most $2$. A homotopy theoretic argument involving the Cube Lemma is then used to prove a  decomposition of $\Omega E$ after localization away from $2$. This implies a decomposition of 
$\Omega(N\#\mathbb{C}P^n)$ after localization away from $2$. One conclusion is that the homotopy effect of the framing can be viewed as an obstruction to a homotopy decomposition of 
$\Omega (N\#\mathbb{C}P^n)$ that holds for all $N$ simultaneously.

In Part \ref{part1}, we apply a ``fibrewise'' version of the above argument to prove Theorem \ref{loopblowupthmintro}. Roughly speaking, as the blow up at a point is simply a connected sum with complex projective space, the general blow up $\widetilde{N}$ can be viewed as a kind of {\it fibrewise connected sum} of $N$ with a {\it fibrwise projective space} over $B$. This view is carried out explicitly by introducing the more general notion of {\it fibrewise surgery} in Section~\ref{sec: semi} and constructing a manifold $P_B\nu$ as a fibrewise projective space over~$B$ in Section~\ref{sec: PBnu}.
Similar to the non-fibrewise case, a fibrewise connected sum is exactly the effect of a {\it fibrewise $0$-surgery}. In Section \ref{sec: blowup} we show that the blow up $\widetilde{N}$ is a fibrewise connected sum $N\mathop{\#}\limits_{B}P_B\nu$, which descends to the usual connected sum $N\# \mathbb{C}P^n$ when $B$ is a point. 
From this point of view, there is a canonical circle bundle over the blow up $\widetilde{N}=N\mathop{\#}\limits_{B}P_B\nu$ such that its total manifold $\widetilde{E}$ is the effect of a {\it fibrewise $1$-surgery} on $N\times S^1$.
This is proved in Section \ref{sec: hatPBnu}. This leads to a homotopy equivalence 
$\Omega\widetilde{N}\simeq S^{1}\times\Omega\widetilde{E}$, so to study the homotopy type of 
$\Omega\widetilde{N}$ it is equivalent to study the  homotopy type of $\Omega\widetilde{E}$. To do this we consider the {\it fibrewise framing} of the fibrewise $1$-surgery as for the usual surgery. It turns out that in this fibrewise case the framing is determined by a map from $S^1$ to the gauge group $\mathcal{G}^{+}(\nu)$ of the normal bundle $\nu$, and its homotopy effect can be detected by a class in the image of $[\Sigma B, SO(2n))] \stackrel{\mathfrak{J}_\ast}{\longrightarrow} [\Sigma B, \Omega^{2n}S^{2n}]$. Notice that when $B$ is a point, this reduces to the $J$-homomorphism in the non-fibrewise case. Under the hypotheses of Theorem~\ref{loopblowupthmintro}, it can be shown that this class will be trivial after localization at the prime $p$. Then, similar to the treatment in~\cite{HT2}, a homotopy theoretic argument involving the Cube Lemma is used to prove a decomposition of $\Omega\widetilde{E}$ after localization at the prime $p$. This implies a decomposition of $\Omega\widetilde{N}$ after localization at the prime $p$, proving Theorem \ref{loopblowupthmintro}. The analysis of the homotopy theory of the fibrewise framing and the fibrewise $1$-surgery will be carried out in Section \ref{sec: homotopy}.

This development of fibrewise surgery is interesting in its own right and may have other applications. We summarize the fibrewise surgery statements used to prove Theorem \ref{loopblowupthmintro}.

\begin{theorem}[Lemma \ref{blowup=sumlemma}, Proposition \ref{tildeEprop} and Examples \ref{blowuppointex} and \ref{S1blowuppointex}]\label{surgerythmintro}\label{method_theorem} 
Let $\widetilde{N}$ be the blow up of a closed manifold $N$ along a submanifold $B$ with a normal bundle $\nu$ of complex dimension $n$. Then $\widetilde{N}$ is a fibrewise connected sum $N\mathop{\#}\limits_{B}P_B\nu$ of $N$ with a fibrewise projective space $P_B\nu$ over $B$. 
Further, there is a canonical circle bundle over $\widetilde{N}$ such that its total manifold $\widetilde{E}$ is the effect of a fibrewise $1$-surgery on $N \times S^1$ with fibrewise framing determined by a class of $\pi_1(\mathcal{G}^{+}(\nu))$. 

In the special case when $B$ is a point, the fibrewise projective space $P_B\nu\cong \mathbb{C}P^n$ is the usual projective space and the blow up $\widetilde{N}$ is the usual connected sum $N\#\mathbb{C}P^n$. Further, the total manifold $\widetilde{E}$ is the effect of a usual $1$-surgery on $N\times S^1$ with framing determined by a class of $\pi_1(SO(2n))$. 
\end{theorem}

\noindent 
\textbf{Results via homotopy in Part \ref{part2}}.

We now turn to Part \ref{part2} of the paper. Unlike Part \ref{part1}, where we study a canonical circle bundle over the blow up from a surgery theory point of view, in Part \ref{part2} we study the blow up $\widetilde{N}$ from a purely homotopy theoretic point of view.

From the data of the blow up construction there is the sphere bundle $S^{2n-1}\stackrel{}{\longrightarrow} \partial V\stackrel{}{\longrightarrow} B$ associated to the normal bundle $\nu$. The standard free action of $S^1$ on $S^{2n-1}$ induces a fibrewise free action on $\partial V$, 
\[
\theta: \partial V\times S^1 \stackrel{}{\longrightarrow} \partial V.
\]
In Section \ref{sec:topblowup}, by applying the Cube Lemma twice, we show the following integral loop decomposition of the blow up $\widetilde{N}$. Let $\iota_c: \partial V\hookrightarrow N_c$ be the inclusion of the boundary.
\begin{theorem}[Theorem \ref{loopNtilde}] 
   \label{loopNtildeintro} 
   Suppose that the composite $\partial V\times S^1 \stackrel{\theta}{\longrightarrow} \partial V\stackrel{\iota_c}{\hookrightarrow}N_c$ is homotopic to the composite $\partial V\times S^1 \stackrel{\pi_1}{\longrightarrow} \partial V\stackrel{\iota_c}{\hookrightarrow}N_c$, where $\pi_1$ the projection onto the first factor. Then 
   there is a homotopy equivalence 
   \[\Omega\widetilde{N}\simeq S^{1}\times\Omega N_{c}\times\Omega\Sigma^{2} F,\] 
   where $F$ is the homotopy fibre of $\iota_c$.
\end{theorem} 

The decomposition for $\Omega\widetilde{N}$ in Theorems~\ref{loopblowupthmintro} 
and~\ref{loopNtildeintro} takes the same form but the hypotheses in each case are different and the 
category in which the conclusion holds is also different. Theorem~\ref{loopblowupthmintro} is proved in the stable range and after suitable localization, based on the study of the blow up from the perspective of geometric topology. Theorem \ref{loopNtildeintro} imposes a homotopy theoretic condition to obtain an integral decomposition without dimension restriction. As an example of Theorem~\ref{loopNtildeintro} we return to focal genus $2$-manifolds.

\begin{example}[Twisted double]\label{tdoubex}
Suppose for the focal genus $2$ manifold $N$ in Example \ref{genus2ex}, $\xi=\nu$ and
\[
N\cong D\nu\cup_g D\nu
\]
with $g$ a self diffeomorphism of the boundary. The manifold $N$ is called a {\it twisted double} of $D\nu\cong V$. In Corollary \ref{doublecoro}, we show that for the blow up $\widetilde{N}$ of the twisted double $N$ of $V$, the hypothesis of Theorem \ref{loopNtildeintro} holds and there is a homotopy equivalence
\[\pushQED{\qed} 
\Omega\widetilde{N}\simeq S^{1}\times\Omega B\times\Omega S^{2n+1}.
\qedhere\popQED \]
 \end{example}  

The purely homotopy theoretic condition on the action $\theta$ in the statement of Theorem~\ref{loopNtildeintro} is investigated systematically in Section \ref{sec: act} based on material in~\cite{BT2}. In particular, when $B$ is a point and the blow up $\widetilde{N}$ becomes the connected sum $N\#\mathbb{C}P^n$, the homotopy theoretic condition can be fully determined. Similarly, when $n$ is even, the same argument also works for $N\#\mathbb{H}P^{\frac{n}{2}}$, the connected sum of $N$ with the quaternionic projective space $\mathbb{H}P^{\frac{n}{2}}$, which can be viewed as a {\it quaternionic blow up} of $N$ at a point \cite{GGS}. These two special cases are studied in Section \ref{sec: stab}. In particular, we improve on a result in \cite{HT2} proved using a mixture of homotopy theory and surgery theory, 
with a proof that now only uses homotopy theory and no surgery theory. Let $N_0$ be the manifold $N$ with an open disk removed.

\begin{theorem}[Corollaries \ref{BJScor} and \ref{BJScor2}]\label{stabthmintro} 
Let $N$ be a connected closed $2n$-dimensional smooth manifold. Let $F$ be the homotopy fibre of the inclusion $S^{2n-1}\hookrightarrow N_0$ of the boundary. The following hold: 
\begin{itemize} 
\item if $n\geq 2$ is even, there is a homotopy equivalence 
\[
\Omega (N\# \mathbb{C}P^n)\simeq S^{1}\times \Omega N_0\times\Omega\Sigma^2 F;
\] 
\item if $n\geq 2$ is odd, after localization away from $2$ there is a homotopy equivalence 
\[
\Omega (N\# \mathbb{C}P^n)\simeq S^{1}\times \Omega N_0\times\Omega\Sigma^2 F;
\]
\item if $n\geq 4$ is even and $n\equiv 0\bmod{48}$, there is a homotopy equivalence
\[
\Omega (N\# \mathbb{H}P^{\frac{n}{2}})\simeq S^{3}\times \Omega N_0\times\Omega\Sigma^4 F;
\]
\item if $n\geq 4$ is even and $n\not\equiv 0\bmod{48}$, after localization away from the primes dividing~$\frac{48}{(48,n)}$ there is a homotopy equivalence 
\[
\Omega (N\# \mathbb{H}P^{\frac{n}{2}})\simeq S^{3}\times \Omega N_0\times\Omega\Sigma^4 F.
\]
\end{itemize}
\end{theorem}

Throughout the paper, $H^\ast(B)$ is used to denote singular cohomology $H^\ast(B;\mathbb{Z})$ with integral coefficients.

\bigskip

\noindent{\bf Acknowledgement} 
Ruizhi Huang was supported in part by the National Natural Science Foundation of China (Grant nos. 12331003, 12288201 and 11801544), the National Key R\&D Program of China (No. 2021YFA1002300), the Youth Innovation Promotion Association of Chinese Academy Sciences, and the ``Chen Jingrun" Future Star Program of AMSS.
 
The authors would like to thank Prof. Jianzhong Pan for bringing reference \cite{LS3} to their attention and to Prof. Yang Su for helpful discussions on some aspects of geometric topology. 
They would also like to thank a referee for comments on the interpretation of Theorem \ref{loopblowupthmintro} and the exposition of its proof.


\section{Conventions, terminology and Mather's Cube Lemma}
\label{part0}

Suppose that all spaces have the homotopy types of $CW$-complexes. 
\medskip

\noindent 
\textbf{Some conventions}.
In this paper, the terms diagram, commutative diagram and homotopy commutative diagram are used with their standard meanings in unstable homotopy theory. 
\begin{itemize}
\item A diagram $\mathcal{D}$ means a collection of spaces $X$, $Y,\ldots$ with a collection of maps among them $f$, $g,\ldots$; 
\item A commutative diagram is a diagram $\mathcal{D}$ such that any two compositions of maps from a space $X$ to a space $Y$ in $\mathcal{D}$ are strictly equal;
\item A homotopy commutative diagram is a diagram $\mathcal{D}$ such that any two compositions of maps from a space $X$ to a space $Y$ in $\mathcal{D}$ are homotopic.
\end{itemize}
For instance, a commutative square (resp. homotopy commutative square) is a square diagram that strictly commutes (resp. homotopically commutes), and a commutative cube (resp. homotopy commutative cube) is a cubical diagram that strictly commutes (resp. homotopically commutes).

For two commutative squares $\mathcal{A}$ and $\mathcal{B}$, a {\it morphism} or {\it map} of commutative squares
\[
F: \mathcal{A}\longrightarrow \mathcal{B}
\]
is a commutative cube  
\begin{gather*}
\begin{aligned}
\xymatrixcolsep{1.5pc}
\xymatrixrowsep{1.5pc}
\xymatrix{
 & A_0 \ar[dl]  \ar[rr]  \ar[dd]|!{[d];[d]}\hole && 
 B_0\ar[dl]  \ar[dd]  \\
A_1\ar[dd]  \ar[rr]  &&
B_1\ar[dd] \\
  & A_2 \ar[dl]  \ar[rr]|!{[r];[r]}\hole  && 
 B_2\ar[dl]    \\
 A_3 \ar[rr]  &&
 B_3, 
}
\end{aligned}
\end{gather*}
where the left and right faces are $\mathcal{A}$ and $\mathcal{B}$ respectively. Further, if $\mathcal{A}$ is a {\it constant square}, that is, $A_i=A$ for any $i$ and some space $A$ and the maps among them are all identity, we simply denote $\mathcal{A}=A$ and $F: A\longrightarrow \mathcal{B}$. Similarly, if $\mathcal{B}=B$ is constant, we denote $F: \mathcal{A}\longrightarrow B$.

Let
\[
\mathcal{A}\stackrel{F}{\longrightarrow} \mathcal{B}\stackrel{G}{\longrightarrow} C
\]
be a sequence of morphisms of commutative squares with $C$ a constant square. If each of the sequences $A_i\longrightarrow B_i\longrightarrow C$ is a fibre bundle, we say the commutative square $\mathcal{B}$ is {\it fibred over the base $C$}. Similarly, let
\[
F\stackrel{H}{\longrightarrow}\mathcal{A}\stackrel{F}{\longrightarrow} \mathcal{B}
\]
be a sequence of morphisms of commutative squares with $F$ a constant square. If each of the sequences $F\longrightarrow A_i\longrightarrow B_i$ is a fibre bundle, we say the commutative square $\mathcal{A}$ is {\it fibred with fibre~$F$}.~\medskip


\noindent 
\textbf{Pushouts, homotopy pushouts and geometric pushouts}. 
There are three types of ``pushout'' that appear in this paper. 
The first is the classical pushout from point-set topology. 
\begin{definition} 
Given pointed maps 
\(f\colon\namedright{X}{}{Y}\) 
and 
\(g\colon\namedright{X}{}{Z}\) 
the \emph{pushout} $C$ is defined by the pointed quotient space 
\[C=(Y\amalg Z)/\sim\] 
where $f(x)\sim g(x)$. Diagrammatically, there is a commutative diagram 
\[\diagram 
        X\rto^-{g}\dto^{f} & Z\dto^{j} \\ 
        Y\rto^-{i} & C  
  \enddiagram\] 
where $i$ and $j$ are the composites 
\(Y\hookrightarrow\namedright{Y\amalg Z}{}{C}\) 
and 
\(Z\hookrightarrow\namedright{Y\amalg Z}{}{C}\) 
respectively. 
\end{definition} 

A pushout has a universal property. Given pointed maps 
\(r\colon\namedright{Y}{}{T}\) 
and 
\(s\colon\namedright{Z}{}{T}\) 
such that $s\circ g=r\circ f$ then there is a unique pointed map 
\(\theta\colon\namedright{C}{}{T}\) 
such that $\theta\circ i=r$ and $\theta\circ j=s$. That is, there is a commutative diagram 
\[\xymatrix{ 
    X\ar[d]^{f}\ar[r]^{g} & Z\ar[d]^{j}\ar@/^/[ddr]^{s} & \\ 
    Y\ar[r]^{i}\ar@/_/[drr]_{r} & C\ar@{.>}[dr]^(0.4){\theta} & \\ 
    & & T. 
}\] 

The second type of pushout is a homotopy pushout. Let $I=[0,1]$ be the unit interval.  
Recall that for any pointed map \(f\colon\namedright{X}{}{Y}\), 
the mapping cylinder $M_{f}$ of $f$ is defined by $M_f=Y\cup_{f} (X\times I)$, 
where $f(x)\sim (x,0)$, and the inclusion $Y\hookrightarrow M_f$ of $Y$ into the base of the mapping cylinder is a homotopy equivalence. 
This homotopy equivalence lets us replace $f$ by the inclusion $\overline{f}: X\stackrel{}{\longrightarrow} M_f$ mapping $X$ homeomorphically to $X\times\{1\}\subseteq M_f$. 

\begin{definition} 
Given pointed maps 
\(f\colon\namedright{X}{}{Y}\) 
and 
\(g\colon\namedright{X}{}{Z}\) 
a \emph{homotopy pushout} is a space $D$ homotopy equivalent to the \emph{double mapping cylinder} $\overline{C}$, the pushout of the two inclusions $\overline{f}: X\stackrel{}{\longrightarrow} M_f$ and $\overline{g}: X\stackrel{}{\longrightarrow} M_g$. Diagrammatically, there is a homotopy commutative diagram 
\[\diagram 
        X\rto^-{g}\dto^{f} & Z\dto^{\overline{j}} \\ 
        Y\rto^-{\overline{i}} & D 
  \enddiagram\] 
where $\overline{i}$ and $\overline{j}$ are the composites 
\(Y\stackrel{k}{\hookrightarrow}\namedright{M_f}{i}{\overline{C}\simeq D}\) 
and 
\(Z\stackrel{k}{\hookrightarrow}\namedright{M_g}{j}{\overline{C}\simeq D}\) 
respectively. 
\end{definition} 
From the definition, the homotopy pushout of $f$ and $g$ is unique up to homotopy equivalence. 
There is a ``semi-universal" property. Given pointed maps 
\(r\colon\namedright{Y}{}{T}\) 
and 
\(s\colon\namedright{Z}{}{T}\) 
such that $s\circ g\simeq r\circ f$ then there is a pointed map 
\(\theta\colon\namedright{D}{}{T}\) 
such that $\theta\circ \overline{i}\simeq r$ and $\theta\circ \overline{j}\simeq s$. That is, there is a homotopy commutative diagram 
\begin{equation}\label{hpushout-univ}
\begin{gathered}
\xymatrix{ 
    X\ar[d]^{f}\ar[r]^{g} & Z\ar[d]^{\overline{j}}\ar@/^/[ddr]^{s} & \\ 
    Y\ar[r]^{\overline{i}}\ar@/_/[drr]_{r} & D\ar@{.>}[dr]^(0.4){\theta} & \\ 
    & & T. 
}
\end{gathered}
\end{equation}
This semi-universal property is weak in the sense that the induced map $\theta$ is not unique in general. 

A pushout is not necessarily a homotopy pushout. However, if one of the structural maps $f$ and~$g$ is the inclusion of a subcomplex or more generally a cofibration, then the pushout of $f$ and $g$ is a homotopy pushout.

The third type of pushout is a geometric pushout, which is defined for manifold gluings.  

\begin{definition}\label{gpushoutdef} 
Let $B_1$ and $B_2$ be two smooth manifolds with nonempty boundary such 
that $\partial B_1=\partial B_2=A$. The {\it geometric pushout} of $B_1$ and $B_2$ is the pushout 
\begin{equation*}
\begin{gathered}
\label{geopushoutdiag}
\xymatrix{
A \ar[r]^{i_1} \ar[d]^{i_2} &
B_1\ar[d]^{j_2} \\
B_2\ar[r]^{j_1} &
Y,
}
\end{gathered}
\end{equation*} 
where $i_1$ and $i_2$ are diffeomorphisms onto the boundaries of $B_1$ and $B_2$. The pushout~$Y$ is a closed smooth manifold after smoothening and both $j_1$ and $j_2$ are embeddings. 
\end{definition} 
Equivalently, $Y$ can be obtained by gluing $B_1$ and $B_2$ along a pair of collars for their boundaries. 
Further, as $i_1$ and $i_2$ are cofibrations, the geometric pushout~$Y$ is a homotopy pushout. 
We summarize the relations among the three types of pushouts as follows.
\begin{lemma}\label{3pushoutlemma}
If one of the structural maps of a pushout is a cofibration, it is a homotopy pushout. 
In particular, a geometric pushout is both a pushout and a homotopy pushout.
$\qqed$
\end{lemma}

\noindent 
\textbf{Mather's Cube Lemma}.
We describe a theorem due to Mather \cite{M}. The statement is actually a special case, 
proved in~\cite[Lemma 3.1]{PT}. 

\begin{theorem}[Mather's Cube Lemma]\label{2cubthm} 
Suppose that there is a homotopy pushout 
\[\diagram 
     A\rto\dto &  B\dto \\ 
     C\rto & D 
  \enddiagram\] 
and a homotopy fibration 
\(\nameddright{D'}{}{D}{h}{Z}\). 
Then composing all the maps in the homotopy pushout with~$h$ and taking homotopy fibres over 
the common base $Z$ gives a homotopy commutative cube
\begin{gather*}
\begin{aligned}
\xymatrixcolsep{1.5pc}
\xymatrixrowsep{1.5pc}
\xymatrix{
 & A^\prime \ar[dl]  \ar[rr]  \ar[dd]|!{[d];[d]}\hole && 
 B^\prime \ar[dl]  \ar[dd]  \\
 C^\prime \ar[dd]  \ar[rr]  &&
 D^\prime \ar[dd] \\
  & A \ar[dl]  \ar[rr]|!{[r];[r]}\hole  && 
 B\ar[dl]    \\
 C \ar[rr]  &&
 D 
}
\end{aligned}
\end{gather*}
that defines $A'$, $B'$ and $C'$, and the cube has the property that the four sides 
are all homotopy pullbacks and the top face is a homotopy pushout.~$\qqed$  
\end{theorem} 

The original statement of the Cube Lemma~\cite[Theorem 25]{M} assumes that 
there is a homotopy commutative cube in which the bottom face is a homotopy pushout, 
the four sides are homotopy pullbacks, and there are certain compatibilities among the homotopies, 
and concludes that the top face is a homotopy pushout. As shown in~\cite[Lemma 3.1]{PT}, the hypothesis that the homotopy commutativity of the cube is due to it being induced by taking fibres over a map 
\(\namedright{D}{h}{Z}\) ensures the compatibility conditions hold.


\part{Blow ups from a surgery theory point of view}
\label{part1}
In this part we study blow ups from the perspective of geometric topology.
In Section \ref{sec: semi} we introduce the notion of {\it fibrewise surgery}, with the {\it fibrewise connected sum} as an example. This is a generalization of the usual surgery and connected sum.
In Section \ref{sec: PBnu} we construct a {\it fibrewise projective space} $P_B\nu$ as a fibrewise copy of $\mathbb{C}P^n$ over $B$. 
In Section \ref{sec: blowup} we review the blow up construction following \cite{LS3} and show that a blow up is the fibrewise connected sum of a manifold~$N$ with the fibrewise projective space $P_B\nu$. 
When $B$ is a point this reduces to the usual connected sum $N\#\mathbb{C}P^n$.
In Section \ref{sec: hatPBnu} we construct a canonical circle bundle over $P_B\nu$, and over the blow up $\widetilde{N}$ as well. We prove that the total manifold of the circle bundle over $P_B\nu$ is a {\it fibrewise sphere} over $B$, and the total manifold $\widetilde{E}$ of the circle bundle over $\widetilde{N}$ is a {\it fibrewise $1$-surgery} of the product manifold $N\times S^1$.
In Section \ref{sec: homotopy} we study the homotopy effect of the {\it fibrewise framing} of the fibrewise $1$-surgery by the classical $J$-homomorphism, and establish a loop homotopy decomposition of the total manifold $\widetilde{E}$. This allows us to prove Theorem \ref{loopblowupthmintro} and its rational version in Theorem~\ref{loopblowupthmintrorational} to obtain a loop homotopy decomposition of the blow up $\widetilde{N}$.
As an application, we prove the extended version of the classical rational dichotomy for the rational homotopy groups of blow ups in Theorem~\ref{dichotomythmintro}.

\section{Fibrewise surgery}
\label{sec: semi} 
We begin by recalling the context for standard surgery theory. 
Let $M$ be an $m$-dimensional closed smooth manifold. An embedding $S^{k}\hookrightarrow M$ is a \emph{framed $k$-embedding} if it extends to an embedding $f\colon D^{m-k}\times S^{k}\hookrightarrow  M$. In this case the original embedding $S^{k}\hookrightarrow M$ is called the \emph{embedding sphere} and $f$ is called the \emph{framing}. In classical geometric topology (see \cite{Ran, Wal} for instance), a {\it $\mathit{k}$-surgery} on $M$ is the surgery removing a framed $\mathit{k}$-embedding $f: D^{m-k}\times S^{k}\hookrightarrow  M$ and replacing it with $S^{m-k-1}\times D^{k+1}$, with {\it effect} the $m$-dimensional closed smooth manifold
\[
M^\prime:=(M\backslash f(\accentset{\circ}{D}^{m-k}\times  S^{k}))\cup_{S^{m-k-1}\times S^{k}} (S^{m-k-1}\times D^{k+1}).
\]

The following lemma is well-known in standard surgery theory. We include a proof which can be generalized to fibrewise surgery theory in the sequel. 
Let $O(m-k)$ be the $(m-k)$-th real orthogonal group. It acts canonically on $\mathbb{R}^{m-k}$ by applying matrix multiplication.
\begin{lemma}\label{frame-lemma}
Let $h: S^{k}\hookrightarrow  M$ be an embedding. 
\begin{itemize}
\item[(1)] Up to an isotopy relative to $h$, any two framings $f_1$ and $f_2$ differ by a bundle isomorphism
\[
f_{1,2}: D^{m-k} \times S^k\longrightarrow D^{m-k}\times S^k
\] 
over $S^k$ with the property that $f_{1,2}(v, t)= (\sigma(t) v, t)$ for some map $\sigma: S^k \rightarrow O(m-k)$. 

Furthermore, when $M$ is oriented and $f_1$ and $f_2$ are orientation-preserving, the map $\sigma$ factors through the subgroup $SO(m-k)$ of $O(m-k)$.
\item[(2)] The effects of surgeries along two isotopic framings relative to $h$ are diffeomorphic.
\end{itemize}
\end{lemma}
\begin{proof}
A framed embedding $D^{m-k}\times S^{k}\hookrightarrow  M$ can be viewed as an embedding of the trivial disk bundle 
\[
D^{m-k}\longrightarrow D^{m-k}\times S^{k}\stackrel{p_2}{\longrightarrow}  S^{k}
\]
into $M$ such that it restricts to the prescribed embedding $S^{k}\stackrel{h}{\hookrightarrow}  M$ on the zero section. 
This means that the framed embedding $D^{m-k}\times S^{k}\hookrightarrow  M$ is exactly a closed tubular neighborhood of the embedding core sphere $S^{k}\stackrel{h}{\hookrightarrow}  M$. 

(1) Suppose that $f_1$ and $f_2$ are two framed embeddings that extend $h$. 
By the uniqueness of fhe tubular neighborhood \cite[Chapter 4, Theorems 5.3 and 6.5]{Hir}\footnote{To avoid confusion, we do not adopt the terminology ``isotopy of tubular neighborhoods'' used in \cite[Chapter~4]{Hir}, which is equivalent to the uniqueness of the tubular neighborhood up to an isotopy of embeddings plus a bundle isomorphism.}, there is an isotopy
\[
F: D^{m-k}\times S^{k}\times I\longrightarrow  M
\]
relative to $\{0\}\times S^{k}\stackrel{h}{\hookrightarrow} M$, such that $F_0=f_1$, $F_1(D^{m-k}\times S^{k})=f_2(D^{m-k}\times S^{k})$, and $F_1=f_2\circ f_{1,2}$ for some bundle isomorphism
\[
\xymatrix{
D^{m-k}\times S^{k} \ar[rr]^{f_{1,2}} \ar[dr]_{p_2} && D^{m-k}\times S^{k} \ar[dl]^{p_2}\\
& S^{k}.
}
\]
Accordingly, $f_1$ is isotopic to $f_2\circ f_{1,2}$ relative to $h$. Since the structural group of a real vector bundle can be reduced to the orthogonal group, by a further isotopy, the bundle isomorphism $f_{1,2}$ can be expressed as 
\[
f_{1,2}(v, t)= (\sigma(t) v, t)
\] 
for some map $\sigma: S^{k}\stackrel{}{\longrightarrow} O(m-k)$. Moreover, when~$M$ is oriented and $f_1$ and $f_2$ are orientation-preserving, the corresponding diffeomorphism $f_{1,2}$ becomes orientation-preserving, and so the map~$\sigma$ reduces to a map $S^{k}\stackrel{}{\longrightarrow} SO(m-k)$.

(2) Suppose that $f_1$, $f_2:  D^{m-k}\times S^{k}\hookrightarrow  M$ are two isotopic framed embeddings relative to the given sphere embedding $S^{k}\stackrel{h}{\hookrightarrow} M$. Then there is an isotopy 
\[
F: f_1(D^{m-k}\times S^{k}) \times I \longrightarrow  M
\]
relative to $h(\{0\}\times S^{k})\stackrel{}{\hookrightarrow} M$ from the inclusion $F_0: f_1(D^{m-k}\times S^{k}) \hookrightarrow M$ to the embedding $F_1: f_1(D^{m-k}\times S^{k})\stackrel{f_2\circ f_1^{-1}}\llarrow f_2(D^{m-k}\times S^{k})\hookrightarrow M$. 
By the ambient tubular neighborhood theorem \cite[Chapter 8, Theorem 1.8]{Hir}, the isotopy $F$ can be extended to an ambient isotopy $\widehat{F}: M\times I \longrightarrow M$ from $\widehat{F}_0={\rm id}$. In particular, $\widehat{F}_1$ is an extension of $F_1$, and its restriction gives a diffeomorphism
\[
\widehat{F}_{1c}: M\backslash f_1(\accentset{\circ}{D}^{m-k}\times  S^{k}) \stackrel{\cong}{\longrightarrow} M\backslash f_2(\accentset{\circ}{D}^{m-k}\times  S^{k})
\]
on the complements extending the diffeomorphism $f_2\circ f_1^{-1}: f_1(S^{m-k-1}\times S^{k})\longrightarrow f_2(S^{m-k-1}\times S^{k})$ on the boundaries. 
Then by the gluing theorem \cite[Chapter, Theorem 2.2]{Hir}, we have a diffeomorphism between the effects of surgeries along $f_1$ and $f_2$:
\[
(M\backslash f_1(\accentset{\circ}{D}^{m-k}\times  S^{k}))\cup_{f_1} (S^{m-k-1}\times D^{k+1})
\stackrel{\widehat{F}_{1c}\cup {\rm id}}{\llarrow}
(M\backslash f_2(\accentset{\circ}{D}^{m-k}\times  S^{k}))\cup_{f_2} (S^{m-k-1}\times D^{k+1}).
\] 
\end{proof}

We now develop a context for fibrewise surgery theory. 
Let $V\stackrel{\pi}{\longrightarrow}B$ be the disk bundle of a rank $(m-l-k)$ real vector bundle $\nu$ over an $l$-dimensional closed manifold $B$. 
Let $B\stackrel{s_0}{\hookrightarrow} V$ be the zero section of $V$.
Let $B\times S^{k}\stackrel{h}{\hookrightarrow} M$ be a {\it $B$-family of embedding $k$-spheres}, that is, for every point $b\in B$ the restriction of $h$ to $\{b\}\times S^{k-1}$ is an embedding $k$-sphere. 
Suppose that $g: V\times S^{k}\hookrightarrow  M$ is an embedding such that $g\circ (s_0\times {\rm id})=h$. We call $g$ a {\it fibrewise framed $\mathit{k}$-embedding} over $B$.
\begin{definition}\label{semi-surdef}
Let $V\stackrel{}{\longrightarrow}B$ be a disk bundle. For a fixed $B$-family of embedding $k$-spheres in $M$, a {\it fibrewise $\mathit{k}$-surgery} on $M$ is the surgery removing a fibrewise framed $\mathit{k}$-embedding $g: V\times S^{k}\hookrightarrow  M$ over $B$ and replacing it with $\partial V\times D^{k+1}$, with {\it effect} the $m$-dimensional closed smooth manifold
\[
M^\prime:=(M\backslash g(\accentset{\circ}{V}\times  S^{k}))\cup_{\partial V\times S^{k}} (\partial V\times D^{k+1}).
\]
\end{definition}

To analyze fibrewise framings we need to use gauge groups. Let $G$ be a topological group. For a principal $G$-bundle $P$ over $B$, the {\it gauge group} of~$P$ is the group of $G$-equivariant automorphisms of~$P$ that fix the base $B$. In the special case of a principal bundle associated to a vector bundle $\nu$, the gauge group of the principal $O(m-l-k)$-bundle acts canonically and fibrewisely on $\nu$, and is isomorphic to the automorphism group of $\nu$ modulo a contractible linear group. We will therefore not distinguish these two groups, and denote them by $\mathcal{G}(\nu)$. Furthermore, if $\nu$ is oriented, the gauge group $\mathcal{G}(\nu)$ has a subgroup $\mathcal{G}^{+}(\nu)$, which is the gauge group of the principal $SO(m-l-k)$-bundle of~$\nu$, and is isomorphic to the group of orientation-preserving automorphisms of $\nu$ modulo a contractible linear group.

The following lemma and its proof is a fibrewise version of Lemma \ref{frame-lemma} and its proof:
\begin{lemma}\label{fib-frame-lemma}
Let $h: B\times S^{k}\hookrightarrow  M$ be a $B$-family of embedding $k$-spheres. 
\begin{itemize}
\item[(1)] Up to an isotopy relative to $h$, any two fibrewise framings $g_1$ and $g_2$ differ by a bundle isomorphism 
\[
g_{1,2}: V \times S^k\longrightarrow V\times S^k
\] 
over $B\times S^k$ with the property that $g_{1,2}(v, t)= (\tau(t) v, t)$ for some map $\tau: S^k \rightarrow \mathcal{G}^{}(\nu)$. 

Furthermore, when $M$ and $B$ are oriented, $\nu$ is an oriented bundle, and $g_1$ and $g_2$ are orientation-preserving, the map $\tau$ factors through the subgroup $\mathcal{G}^{+}(\nu)$ of $\mathcal{G}^{}(\nu)$. 
\item[(2)] The effects of fibrewise surgeries along two isotopic fibrewise framings relative to $h$ are diffeomorphic.
\end{itemize}
\end{lemma}
\begin{proof}
A fibrewise framed embedding $V\times S^{k}\hookrightarrow  M$ can be viewed as an embedding of the disk bundle 
\[
D^{m-l-k}\longrightarrow V\times S^{k}\stackrel{\pi\times {\rm id}}{\llarrow} B\times S^{k}
\] 
into $M$ such that it restricts to the given embedding $B\times S^{k}\hookrightarrow  M$ on the zero section. 
This means that the fibrewise framed embedding $V\times S^{k}\hookrightarrow  M$ is exactly a closed tubular neighborhood of the embedding core manifold $B\times S^{k}\hookrightarrow  M$. 

(1) Suppose that $g_1$ and $g_2$ are two fibrewise framed embeddings that extend $h$. 
By the uniqueness of the tubular neighborhood \cite[Chapter 4, Theorems 5.3 and 6.5]{Hir}, there is an isotopy
\[
G: V\times S^{k}\times I\longrightarrow  M
\]
relative to $B\times S^{k}\stackrel{h}{\hookrightarrow} M$, such that $G_0=g_1$, $G_1(V\times S^{k})=g_2(V\times S^{k})$, and $G_1=g_2\circ g_{1,2}$ for some bundle isomorphism
\[
\xymatrix{
V\times S^{k} \ar[rr]^{g_{1,2}} \ar[dr]_{\pi\times {\rm id}} && V\times S^{k} \ar[dl]^{\pi\times {\rm id}}\\
& B\times S^{k}.
}
\]
Accordingly, $g_1$ is isotopic to $g_2\circ g_{1,2}$ relative to $h$. By a further isotopy, the bundle isomorphism $g_{1,2}$ can be expressed as 
\[
g_{1,2}(v, t)= (\tau(t) v, t)
\] 
for some map $\tau: S^{k}\stackrel{}{\longrightarrow} \mathcal{G}(\nu)$. Moreover, when $M$ and $B$ are oriented, $\nu$ is an oriented bundle, and $g_1$ and $g_2$ are orientation-preserving, the corresponding diffeomorphism $g_{1,2}$ becomes orientation-preserving, and so the map $\tau$ reduces to a map to the gauge group $\mathcal{G}^{+}(\nu)$.

(2) Suppose that $g_1$, $g_2:  V\times S^{k}\hookrightarrow  M$ are two isotopic fibrewise framed embeddings relative to the prescribed embedding $B\times S^{k}\stackrel{h}{\hookrightarrow} M$. Then there is an isotopy 
\[
G: g_1(V\times S^{k}) \times I \longrightarrow  M
\]
relative to $h(B\times S^{k})\stackrel{}{\hookrightarrow} M$ from the inclusion $G_0: g_1(V\times S^{k}) \hookrightarrow M$ to the embedding $G_1: g_1(V\times S^{k})\stackrel{g_2\circ g_1^{-1}}\llarrow g_2(V\times S^{k})\hookrightarrow M$. 
By the ambient tubular neighborhood theorem \cite[Chapter~8, Theorem 1.8]{Hir}, the isotopy $G$ can be extended to an ambient isotopy $\widehat{G}: M\times I \longrightarrow M$ from $\widehat{G}_0={\rm id}$. In particular, $\widehat{G}_1$ is an extension of $G_1$, and its restriction gives a diffeomorphism
\[
\widehat{G}_{1c}: M\backslash g_1(\accentset{\circ}{V}\times  S^{k}) \stackrel{\cong}{\longrightarrow} M\backslash g_2(\accentset{\circ}{V}\times  S^{k})
\]
on the complements extending the diffeomorphism $g_2\circ g_1^{-1}: g_1(\partial V\times S^{k})\longrightarrow g_2(\partial V\times S^{k})$ on the boundaries. 
Then by the gluing theorem \cite[Chapter, Theorem 2.2]{Hir}, we have a diffeomorphism between the effects of fibrewise surgeries along $g_1$ and $g_2$:
\[
(M\backslash g_1(\accentset{\circ}{V}\times  S^{k}))\cup_{g_1} (\partial V\times D^{k+1})
\stackrel{\widehat{G}_{1c}\cup {\rm id}}{\llarrow}
(M\backslash g_2(\accentset{\circ}{V}\times  S^{k}))\cup_{g_2} (\partial V\times D^{k+1}).
\] 
\end{proof}  

Observe that when $l=0$ and $B$ is a point, a fibrewise $k$-surgery is a standard $k$-surgery, and the map $\tau$ measuring the fibrewise difference of two framings reduces to a map from $S^k$ to $O(m-k)$, or to $SO(m-k)$ in the oriented case.

We will encounter two special fibrewise surgeries, one of which is the fibrewise $0$-surgery in the following example.
\begin{example}[Fibrewise connected sum]
\label{fibre0surex}
Let $M_1$ and $M_2$ be two $m$-dimensional closed manifolds. Let $V\stackrel{}{\longrightarrow}B$ be the disk bundle of a rank $(m-l)$ real vector bundle over an $l$-dimensional closed manifold $B$. Suppose that $g_1: V\stackrel{}{\hookrightarrow} M_1$ and $g_2: V\stackrel{}{\hookrightarrow} M_2$ are two fibrewise framed $0$-embeddings. Define the {\it fibrewise connected sum} of $M_1$ and $M_2$ over $B$ to be the manifold
\[
M_1\mathop{\#}\limits_{B}M_2:= (M_1 \backslash g_1(\accentset{\circ}{V}))\cup_{\partial V}(\partial V\times [0,1])\cup_{\partial V}(M_2 \backslash g_2(\accentset{\circ}{V})) 
\] 
obtained by gluing the two ends of $\partial V\times [0,1]$ to $M_{1}\backslash g_1(\accentset{\circ}{V})$ and $M_2 \backslash g_2(\accentset{\circ}{V})$ by $g_{1}$ and $g_{2}$. 
It is then straightforward to check that $M_1\mathop{\#}\limits_{B}M_2$ is exactly the effect of the fibrewise $0$-surgery on $M_1\amalg M_2$ determined by the fibrewise framed $0$-embedding $g=g_1\amalg g_2: V\times S^{0}\stackrel{}{\hookrightarrow} M_1 \amalg M_2$. In particular, when $B$ is a point then $M_1\mathop{\#}\limits_{B}M_2$ is the usual connected sum $M_1\# M_2$. 

Further, we can define a complex 
\[
 M_1\mathop{\vee}\limits_{B} M_2:=  M_1 \cup_{g_1}(B\times [0,1])\cup_{g_2} M_2
\]
by gluing the two ends of $B\times [0,1]$ to $M_1$ and $M_2$ through the restriction of $g_1$ and $g_2$ to their corresponding zero-sections. It can be viewed as a {\it generalized wedge sum} of $M_1$ and $M_2$ over $B$.
Similarly, define a complex
\[
M_1\mathop{\#}\limits_{V}M_2:= M_1 \cup_{g_1}(V\times [0,1])\cup_{g_2} M_2
\]
by gluing the two ends of $V\times [0,1]$ to $M_1$ and $M_2$ through $g_1$ and $g_2$. 
Then we can define a composite
\[
q: M_1\mathop{\#}\limits_{B}M_2 \stackrel{b}{\hookrightarrow} M_1\mathop{\#}\limits_{V}M_2 \stackrel{\mathfrak{q}}{\longrightarrow}  M_1\mathop{\vee}\limits_{B} M_2, 
\]
where $b$ is the obvious inclusion map and $\mathfrak{q}$ is induced by the bundle projection $V\rightarrow B$ on the neck (that is, on the $V\times [0,1]$ part). 
The map $q$ can be viewed as the {\it fibrewise pinch map} of the fibrewise connected sum $M_1\mathop{\#}\limits_{B}M_2$. In particular, when $B$ is a point, this is the usual pinch map for a connected sum.
\end{example}
\section{Fibrewise projective space}
\label{sec: PBnu}
In standard surgery theory there is the connected sum $M\# \mathbb{C}P^{n}$ of a manifold $M$ with the complex projective space $\mathbb{C}P^{n}$. To formulate a fibrewise analogue as in Example \ref{fibre0surex} we need a ``fibrewise projective space''. We will construct such a space this in this section. In what follows, we will often use the same notation for a bundle and its total space, by a slight abuse of notation. 
 
Let $B$ be an $l$ dimensional connected closed smooth manifold. Let $\nu$ be an $n$-dimensional complex bundle over $B$. To be consistent with the notation in Section \ref{sec: semi}, let
\[\label{pinueq}
D^{2n}\stackrel{}{\longrightarrow} V \stackrel{\pi_\nu}{\longrightarrow} B
\] 
be the disk bundle of $\nu$. Restricting to the corresponding sphere bundle $S\nu=\partial V$ gives a commutative diagram of fibre bundles 
\begin{equation}
\begin{gathered}
\label{Snudiag}
\xymatrix{
S^{2n-1}\ar[r]^{i_\nu} \ar[d]^{\iota} &  \partial V \ar[r]^{s_\nu} \ar[d]^{\iota_{\nu}} & B \ar@{=}[d] \\
D^{2n} \ar[r]   &  V   \ar[r]^{\pi_\nu}               & B
}
\end{gathered}
\end{equation}  
where $\iota$ and $\iota_{\nu}$ are the inclusions, $s_{\nu}$ is defined as $\pi_{\nu}\circ\iota_{\nu}$, 
and $i_{\nu}$ is the inclusion of the fibre. Since~$\nu$ is a complex bundle there is a 
corresponding projective bundle $P\nu$ satisfying a commutative diagram of fibre bundles
\begin{equation}
\begin{gathered}
\label{SPnudiag}
\xymatrix{
S^1 \ar@{=}[r] \ar[d] & S^1 \ar[d]^{}\\
S^{2n-1}\ar[r]^{i_\nu} \ar[d]^{q_0} &  \partial V \ar[r]^{s_\nu} \ar[d]^{q} & B \ar@{=}[d] \\
\mathbb{C}P^{n-1} \ar[r]^{i_p}   &  P\nu   \ar[r]^{p_\nu}               & B,
}
\end{gathered}
\end{equation} 
where $q_{0}$ is the standard quotient, $q$ is the fibrewise quotient, $p_\nu$ is 
the quotient of $s_\nu$, and $i_p$ is the inclusion of the fibre.  

For any $x\in B$, let $P\nu_x$ be the space of all complex lines through the origin in the fibre $\nu_x$ of~$\nu$ over $x$. Then $\nu_x\cong \mathbb{C}^n$, $P\nu_x\cong \mathbb{C}P^{n-1}$, and 
in terms of points, 
\[P\nu=\{(x,L_{x})\mid x\in B\ \mbox{and $L_{x}\in P\nu_x$}\}.\] 
There is a canonical complex line bundle $\lambda_{\nu}$ over $P\nu$ defined by 
\[\lambda_{\nu}=\{(x,L_{x},z)\mid (x,L_{x})\in P\nu\ \mbox{and $z\in L_{x}$}\}.\] 
From the construction, the circle bundle $S\lambda_{\nu}$ of $\lambda_{\nu}$ and the bundle \(\namedright{\partial V}{q}{P\nu}\) coincide: 
\begin{equation}\label{Snu=Slambdaeq}
\partial V= S\lambda_\nu.
\end{equation}
Therefore $q$ factors as
\begin{equation}\label{qfactoreq}
q: \partial V\stackrel{=}{\longrightarrow} S\lambda_{\nu}\stackrel{\iota_\lambda}{\longrightarrow} D\lambda_\nu \stackrel{\pi_\lambda}{\longrightarrow} P\nu,
\end{equation} 
where $D\lambda_{\nu}$ is the disk bundle of $\lambda_{\nu}$, $\pi_\lambda$ is the projection with fibre $D^2$, and $\iota_\lambda$ is the inclusion of the boundary. The space $D\lambda_{\nu}$ 
should be regarded as a thickening of $P\nu$ that will turn certain pushouts into geometric pushouts. 
The fact that $\pi_{\lambda}$ is a homotopy equivalence will be recorded for later. 

\begin{lemma} 
   \label{pilambdaequiv} 
   The map 
   \(\namedright{D\lambda_{\nu}}{\pi_{\lambda}}{P\nu}\) 
   is a homotopy equivalence.~$\qqed$ 
\end{lemma} 

Let $\eta$ be the canonical complex line bundle over $\mathbb{C}P^{n-1}$. Let $D\eta$ and 
$S\eta$ be the corresponding disk and sphere bundles for $\eta$, and note that $S\eta= S^{2n-1}$.

\begin{lemma}\label{rhonulemma}
The composite $\rho_\nu: D\lambda_\nu  \stackrel{\pi_\lambda}{\longrightarrow} P\nu \stackrel{p_\nu}{\longrightarrow} B$ is a fibre bundle with fibre isomorphic to the disk bundle $D\eta$:
\[\label{rhonueq}
D\eta\stackrel{\jmath_\nu}{\longrightarrow}D\lambda_\nu \stackrel{\rho_\nu}{\longrightarrow} B,
\]
and further its boundary is the sphere bundle $s_\nu: \partial V\rightarrow B$ in~(\ref{SPnudiag}), that is, $s_\nu$ factors as
\[
s_\nu: \partial V\stackrel{\iota_\lambda}{\longrightarrow}  D\lambda_\nu\stackrel{\rho_\nu}{\longrightarrow}  B.
\]
\end{lemma}
\begin{proof}
By definition, the fibre of $\rho_\nu$ over a point $x\in B$ consists of the points $(L_x, z)\in P\nu_x \times D^2$ such that $z\in L_x$. This is isomorphic to the disk bundle $D\eta$ of $\eta$. Restricting to the sphere bundle $\partial V=S\lambda_\nu$, the composite 
\[
\partial V\stackrel{=}{\longrightarrow} S\lambda_{\nu}\stackrel{\iota_\lambda}{\longrightarrow} D\lambda_\nu \stackrel{\pi_\lambda}{\longrightarrow} P\nu \stackrel{p_\nu}{\longrightarrow} B
\]
is equal to $p_\nu\circ q=s_\nu$ by (\ref{qfactoreq}) and (\ref{SPnudiag}). 
\end{proof}
From Lemma \ref{rhonulemma} and  (\ref{qfactoreq}), we see that the morphism of fibre bundles over $B$ in (\ref{SPnudiag}) can be factored as a commutative diagram of fibre bundles
\begin{equation}
\begin{gathered}
\label{bVtoDlamdiag}
\xymatrix{
S^{2n-1}\ar[r]^{i_\nu} \ar[d]^{\iota_\eta} &  \partial V \ar[r]^{s_\nu} \ar[d]^{\iota_\lambda} & B \ar@{=}[d] \\
D\eta\ar[r]^{\jmath_\nu} \ar[d]^{\pi_\eta}   & D\lambda_\nu   \ar[r]^{\rho_\nu}  \ar[d]^{\pi_\lambda}               & B\ar@{=}[d]\\
\mathbb{C}P^{n-1} \ar[r]^{i_p}  & P\nu \ar[r]^{p_\nu}  & B,
}
\end{gathered}
\end{equation} 
where $\pi_\eta$ is the projection of the disk bundle of $\eta$, $\iota_\eta$ is the inclusion of the boundary, and $\pi_\eta\circ\iota_\eta=q_0$.
In particular, this implies that $\iota_\lambda$ is fibred over $B$.

Now we construct a fibrewise projective space $P_B\nu$. There is a canonical geometric pushout 
\begin{equation}
\begin{gathered}
\label{PBnudiag1}
\xymatrix{
S^{2n-1}  \ar[r]^{\iota_\eta}        \ar[d]^{\iota}      &   D\eta  \ar[d]^{}   \\
D^{2n} \ar[r]^{}      & \mathbb{C}P^{n}. 
}
\end{gathered}
\end{equation} 
Define the $(l+2n)$-dimensional closed manifold $P_{B}\nu$ and the maps $j_{\nu}$ and $j_{\lambda}$ 
by the geometric pushout 
\begin{equation}
\begin{gathered}
\label{PBnudiag}
\xymatrix{
\partial V \ar[r]^{\iota_\lambda}   \ar[d]^{\iota_\nu} &  D\lambda_\nu \ar[d]^{\jmath_\nu}  \\
V \ar[r]^{\jmath_\lambda}                                         &   P_B\nu. 
}
\end{gathered}
\end{equation}

\begin{lemma}\label{kappanulemma} 
The pushout in (\ref{PBnudiag}) is fibred over $B$ with fibres given by the 
pushout~(\ref{PBnudiag1}). In particular, there exists a fibre bundle
\[\label{kappanueq}
\mathbb{C}P^{n}\stackrel{}{\longrightarrow} P_B\nu \stackrel{\kappa_\nu}{\longrightarrow} B
\]
such that the bundle maps $\pi_\nu$ and $\rho_\nu$ respectively factor as
\[
\pi_\nu:  V\stackrel{\jmath_{\lambda}}{\longrightarrow}  P_B\nu \stackrel{\kappa_\nu}{\longrightarrow}B, \ \ \ ~ 
\rho_\nu:  D\lambda_\nu\stackrel{\jmath_{\nu}}{\longrightarrow}  P_B\nu \stackrel{\kappa_\nu}{\longrightarrow} B.
\]

\end{lemma}
\begin{proof}
By~(\ref{Snudiag}), $\pi_{\nu}\circ\iota_{\nu}=s_{\nu}$ and, by~(\ref{bVtoDlamdiag}), 
$\rho_{\nu}\circ\iota_{\lambda}= s_{\nu}$. The universal property of a pushout therefore implies that exists a unique map $\kappa_\nu: P_B\nu\rightarrow  B$ 
such that $\pi_{\nu}=\kappa_{\nu}\circ j_{\lambda}$ and $\rho_{\nu}=\kappa_{\nu}\circ j_{\nu}$. 
Further, by~(\ref{Snudiag}), $\iota_\nu$ is fibred over $B$ with $\iota$ as fibre and, by (\ref{bVtoDlamdiag}), 
$\iota_\lambda$ is fibred over $B$ with $\iota_\eta$ as fibre. Therefore 
\(\namedright{P_{B}\nu}{\kappa_{\nu}}{B}\) 
is a fibre bundle with fibre $\mathbb{C}P^n$. 
\end{proof}

\begin{definition}\label{PBnudef}
Let $\nu$ be an $n$-dimensional complex bundle over an $l$-dimensional closed manifold~$B$.
The $(l+2n)$-dimensional closed manifold $P_B\nu$ defined in (\ref{PBnudiag}) is called the {\it fibrewise projective space} over $B$.
\end{definition}

Let $N$ be an $m$-dimensional closed manifold such that $m=l+2n$.
Suppose $i: V\stackrel{}{\hookrightarrow} N$ is an embedding of a disk bundle $V\longrightarrow B$ over an $l$-dimensional closed manifold $B$. Pairing this with the embedding 
\(\namedright{V}{\jmath_{\lambda}}{P_B{\nu}}\), 
then by Example \ref{fibre0surex} we can define the fibrewise connected sum
\begin{equation}\label{fsumeq}
N\mathop{\#}\limits_{B}P_B\nu:= (N \backslash i(\accentset{\circ}{V}))\cup_{\partial V}(\partial V\times [0,1])\cup_{\partial V}(P_B\nu \backslash \jmath_\lambda(\accentset{\circ}{V}))\cong  (N \backslash i(\accentset{\circ}{V}))\cup_{\partial V} D\lambda_\nu.
\end{equation}
We call it the {\it fibrewise $\mathbb{C}P^n$-stabilization} of $N$. Indeed, when $l=0$ and $B$ is point, $V=D^{2n}$ and by Lemma \ref{kappanulemma} $P_B\nu$ reduces to $\mathbb{C}P^n$. Then
\[
N\mathop{\#}\limits_{\ast }\mathbb{C}P^n= (N \backslash \accentset{\circ}{D^{2n}})\cup_{S^{2n-1}}(S^{2n-1}\times [0,1])\cup_{S^{2n-1}}(\mathbb{C}P^n \backslash \accentset{\circ}{D^{2n}})\cong  (N \backslash \accentset{\circ}{D^{2n}})\cup_{S^{2n-1}}D\eta
\]
is the usual connected sum $N\#\mathbb{C}P^n$, which is the {\it $\mathbb{C}P^n$-stabilization} of $N$ \cite{K2}.


\section{A blow up is a fibrewise $\mathbb{C}P^n$-stabilization}
\label{sec: blowup}
In this section,
we first recall the description in~\cite[Section 2]{LS3} of the blow up construction and then show that it is essentially a fibrewise $\mathbb{C}P^n$-stabilization.

Let $B\hookrightarrow N$ be a codimension $2n$ embedding of connected closed oriented smooth manifolds. 
Suppose that the normal bundle $\nu$ of the embedding has been given a complex structure. By abuse of notation this $n$-dimensional complex bundle is still denoted by $\nu$.
By the classical tubular neighborhood theorem \cite[Theorem 11.1]{MS}, there is a closed neighborhood $V$ of $B$ in $N$ and a diffeomorphism between $V$ and the disk bundle $D\nu$, restricting to a diffeomorphism between $\partial V$ and the sphere bundle $S\nu$, and whose zero section is $B$.
We shall not distinguish between $V$ and $D\nu$ in what follows.
Let $N_c$ be the closure of the complement of $V$ in $N$. Then $\partial N_c=\partial V$, and there is a geometric pushout
\begin{equation}
\begin{gathered}
\label{Npushoutdiag}
\xymatrix{
\partial V \ar[r]^{\iota_\nu} \ar[d]^{\iota_c} &
V\ar[d]^{i} \\
N_c\ar[r] &
N,
}
\end{gathered}
\end{equation} 
where $i$ is the embedding, and $\iota_c$ and $\iota_\nu$ are the inclusions of the respective boundaries.

Replacing the inclusion of the boundary $\partial V\stackrel{\iota_{\nu}}{\longrightarrow} V$ 
in~(\ref{Npushoutdiag}) with the projectivization $\partial V\stackrel{q}{\longrightarrow} P\nu$, define 
the space $\widetilde{N}$ and the maps $j$ and $r$ by the pushout  
\begin{equation}
\begin{gathered}
\label{blowuppushoutdiag}
\xymatrix{
\partial V \ar[r]^{q} \ar[d]^{\iota_c} &
P\nu\ar[d]^{j} \\
N_c\ar[r]^{r} &
\widetilde{N}. 
}
\end{gathered}
\end{equation} 
The space $\widetilde{N}$ is the {\it topological blow up of $N$ along $B$}. Observe that this pushout 
determines the homeomorphism type of the blow up.

Notice that in (\ref{blowuppushoutdiag}), the map $q$ is not the inclusion of a 
boundary so the pushout defining $\widetilde{N}$ is not a geometric pushout. 
To recover the geometry, consider Diagram (\ref{bVtoDlamdiag}) as a refinement of Diagram~(\ref{SPnudiag}). 
In particular, the projection $q$ factors as
\[
q: \partial V\stackrel{\iota_\lambda}{\longrightarrow} D\lambda_\nu \stackrel{\pi_\lambda}{\longrightarrow} P\nu,
\]
where $\iota_\lambda$ is the inclusion of the boundary of the disk bundle $D\lambda_\nu$ and the bundle projection $\pi_\lambda$ of $D\lambda_\nu$ is a homotopy equivalence.
Accordingly, in Diagram (\ref{blowuppushoutdiag}) we can replace $q$ by $\iota_\lambda$ to obtain an alternative definition of the blow up of $N$ along $B$ as a geometric pushout
\begin{equation}
\begin{gathered}
\label{blowuppushoutdiag2}
\xymatrix{
\partial V \ar[r]^{\iota_\lambda} \ar[d]^{\iota_c} &
D\lambda_\nu       \ar[d]^{r_j} \\
N_c\ar[r]^{r} &
\widetilde{N}  
}
\end{gathered}
\end{equation} 
that defines the manifold $\widetilde{N}$ and the maps $r_{j}$ and $r$. The smooth manifold $\widetilde{N}$ is called the {\it geometric blow up of $N$ along $B$}.

\begin{remark}\label{top-geo-blow-rmk} 
The topological and geometric blow ups in~(\ref{blowuppushoutdiag}) and~(\ref{blowuppushoutdiag2}) are different. Nevertheless, the homotopy equivalence~$\pi_{\lambda}$ implies that they are canonically homotopy equivalent, and $r_j$ is homotopic to the composite 
\[\label{rjdef}
D\lambda_\nu\stackrel{\pi_\lambda}{\longrightarrow} P\nu\stackrel{j}{\longrightarrow} \widetilde{N}.
\] 
In terms of the paper, Part \ref{part1} focuses on the 
geometric blow up in~(\ref{blowuppushoutdiag2}), while Part \ref{part2} focuses on the 
topological blow up in~(\ref{blowuppushoutdiag}). 
However, in the Introduction of the paper, we do not distinguish these two blow ups, as all the results are about homotopy types. 
\end{remark}

Diagram~(\ref{blowuppushoutdiag2}) is equivalent to saying that the geometric blow up is defined as 
$\widetilde{N}:=N_c\cup_{\partial V}D\lambda_\nu$, where the boundaries of $N_c$ and $D\lambda_\nu$ are glued together. This determines the diffeomorphism type of the geometric blow up. 
Further, up to a diffeomorphism thickening the collars of the boundaries, $\widetilde{N}$ may be regarded as being obtained by gluing together the boundaries of $N_{c}$ and $D\lambda_{\nu}$ along 
the ends of $\partial V\times [0,1]$. 
By (\ref{fsumeq}), the latter is exactly the definition of the fibrewise 
connected sum $N\mathop{\#}\limits_{B} P_{B}\nu$ formed from the embeddings 
\(\namedright{V}{i}{N}\) 
and 
\(\namedright{V}{\jmath_{\lambda}}{P_B{\nu}}\). 

\begin{lemma}\label{blowup=sumlemma}
The geometric blow up $\widetilde{N}$ is diffeomorphic to the fibrewise connected sum of $N$ with $P_B\nu$,
\[\hspace{2.8cm}
\widetilde{N}\cong N\mathop{\#}\limits_{B}P_B\nu \cong  (N \backslash i(\accentset{\circ}{V}))\cup_{\partial V}(\partial V\times [0,1])\cup_{\partial V}(P_B\nu \backslash \jmath_{\lambda}(\accentset{\circ}{V})). 
\hspace{2.8cm}\Box
\]
\end{lemma}
From now on, we will not distinguish the geometric blow up $\widetilde{N}$ and the fibrewise connected sum $N\mathop{\#}\limits_{B}P_B\nu$.

\begin{example}[Blow up at a point]\label{blowuppointex}
Suppose $B=\ast$ is a point in the above construction. The diagrams (\ref{Npushoutdiag}), (\ref{blowuppushoutdiag}) and (\ref{blowuppushoutdiag2}) respectively become
\[
\begin{gathered}
\label{sumexdiag}
\xymatrix{
S^{2n-1} \ar[r] \ar[d]^{} &  D^{2n}\ar[d]    & S^{2n-1} \ar[r]^{q_0} \ar[d]^{} & \mathbb{C} P^{n-1}\ar[d]^{}   &  S^{2n-1} \ar[r]^{\iota_\eta} \ar[d]^{} & D\eta\ar[d]^{}
\\
N_0\ar[r] &  N,                                       & N_0\ar[r]^{} &\widetilde{N},                               & N_0\ar[r]^{} &\widetilde{N},
}
\end{gathered}
\]
where $N_0$ is the manifold with a small disc removed, $\eta$ is the canonical complex line bundle of $\mathbb{C}P^{n-1}$, and $\iota_\eta$ is the inclusion of the sphere bundle $S\eta\cong S^{2n-1}$ into the disk bundle $D\eta$.
From the definition of the connected sum of manifolds, we see that the geometric and topological blow ups at a point are diffeomorphic and 
\[
\widetilde{N}\cong N\# \mathbb{C}P^n.
\]
Hence, we will not distinguish them.

The special case of a blow up at a point was studied by Duan \cite{D1} from a geometric point of view and by the authors~\cite{HT2} from a homotopy theoretic point of view. 
\end{example}


\section{A canonical circle bundle over a blow up}
\label{sec: hatPBnu}
In this section, we construct a canonical circle bundle over the geometric blow up $\widetilde{N}=N\mathop{\#}\limits_{B}P_B\nu$ of $N$ along~$B$ defined by (\ref{blowuppushoutdiag2}), and show that the total manifold of this circle bundle is the effect of a special fibrewise $1$-surgery on a product manifold. 

To construct a circle bundle we look for an appropriate map 
\(\namedright{\widetilde{N}}{}{\mathbb{C}P^{\infty}}\). 
This will be done by progressing through similar maps for intermediate spaces. From~(\ref{Npushoutdiag}), 
(\ref{PBnudiag}), (\ref{blowuppushoutdiag}) and (\ref{blowuppushoutdiag2}) respectively, there are pushouts 
\begin{equation} 
  \label{4pos} 
  \diagram 
    \partial V\rto^-{\iota_{\nu}}\dto^{\iota_{c}} & V\dto 
        & \partial V\rto^-{\iota_{\lambda}}\dto^{\iota_{\nu}} & D\lambda_{\nu}\dto^{\jmath_{\nu}} 
        & \partial V\rto^-{q}\dto^{\iota_{c}} & P\nu\dto^{j} 
        &  \partial V\rto^-{\iota_{\lambda}}\dto^{\iota_{c}} & D\lambda_{\nu}\dto^{r_{j}} \\ 
    N_{c}\rto & N & V\rto^-{\jmath_{\lambda}} & P_{B}\nu & N_{c}\rto^-{r} & \widetilde{N} & N_{c}\rto^-{r} & \widetilde{N}.  
  \enddiagram 
\end{equation}  
From left to right, the first, second and fourth are geometric pushouts while the third is a topological pushout. 
The fourth is intended as a geometric pushout replacement for the third topological pushout, noting 
there is a homotopy equivalence 
\(\namedright{D\lambda_{\nu}}{\pi_{\lambda}}{P\nu}\). 
The second is a modification of the first by replacing $\iota_{c}$ by $\iota_{\lambda}$, and the fourth 
is a modification of the first by replacing $\iota_{\nu}$ by $\iota_{\lambda}$. We will show there is a canonical map 
\(\namedright{P\nu}{e}{\mathbb{C}P^{\infty}}\), 
use the homotopy equivalence $\pi_{\lambda}$ to translate this to a canonical map 
\(\namedright{D\lambda_{\nu}}{}{\mathbb{C}P^{\infty}}\), 
then use the pushout property of the second and fourth pushouts to obtain ``canonical" maps 
\(\namedright{P_{B}\nu}{}{\mathbb{C}P^{\infty}}\) 
and 
\(\namedright{\widetilde{N}}{}{\mathbb{C}P^{\infty}}\). 
The latter will produce the circle bundle over the geometric blow up $\widetilde{N}$ that we are after. 
This circle bundle will then be examined in terms of the identification of $\widetilde{N}$ as 
the fibrewise connected sum $N\conn_{B} P_{B}\nu$. 

We first produce a map 
\(\namedright{P\nu}{}{\mathbb{C}P^{\infty}}\). 
Consider the projective bundle 
$\mathbb{C}P^{n-1}\stackrel{i_p}{\longrightarrow} P\nu\stackrel{p_\nu}{\longrightarrow} B$ 
of $\nu$. By the standard argument for projective bundles \cite[Section 17.2]{Hus},
the canonical complex line bundle $\lambda_\nu$ of $P\nu$ is a direct summand of the pullback bundle $p_\nu^\ast(\nu)$, and the Euler class $e:=e(\lambda_\nu)\in H^2(P\nu)$ has the property that its restriction to each fibre of the projective bundle is the generator $c\in H^2(\mathbb{C}P^{n-1})$, where $c$ is the Euler class of the canonical line bundle $\eta$ of $\mathbb{C}P^{n-1}$. The Leray-Hirsch theorem \cite[Theorem 1.1 in Section 17.1]{Hus} then implies that there is an isomorphism of $H^\ast(B)$-modules
\begin{equation}\label{LH-Pnueq}
H^\ast(P\nu) \cong H^\ast(B)\{1, e, \cdots, e^{n-1}\}.
\end{equation} 
Abusing notation, the class $e\in H^{2}(P\nu)$ is represented by a map 
\(e\colon\namedright{P\nu}{}{\mathbb{C}P^{\infty}}\). 

Compare this to the projective bundle $\mathbb{C}P^{n-1}\stackrel{i_p}{\longrightarrow} P_\nu \stackrel{p_\nu}{\longrightarrow} B$ in~(\ref{SPnudiag}).
Since the Euler class $e\in H^2(P_\nu)$ restricts to a generator $c\in H^2(\mathbb{C}P^{n-1})$ on each fibre, 
there is a commutative diagram 
\begin{equation} 
  \label{ecdgrm} 
  \diagram 
       \mathbb{C}P^{n-1}\rto^-{i_{p}}\drto_{c} & P\nu\dto^{e} \\ 
       & \mathbb{C}P^{\infty}. 
  \enddiagram 
\end{equation}

\begin{lemma} 
   \label{esphere} 
   There is a map 
   \(\namedright{S^{2}}{}{P\nu}\) 
   with the property that the composite 
   \(\nameddright{S^{2}}{}{P\nu}{e}{\mathbb{C}P^{\infty}}\) 
   induces an isomorphism in degree $2$ cohomology. 
\end{lemma} 

\begin{proof}
Consider the composite 
\(\namedddright{S^{2}}{a'}{\mathbb{C}P^{n-1}}{i_{p}}{P\nu}{e}{\mathbb{C}P^{\infty}}\), 
where $a'$ is the inclusion of the bottom cell. By~(\ref{ecdgrm}), $e\circ i_{p}$ equals $c$. 
Since $a'$ and $c$ both induce an isomorphism in degree~$2$ cohomology, so does the 
composite $e\circ i_{p}\circ a'$. Therefore, taking $a=i_{p}\circ a'$ we obtain a map 
\(\namedright{S^{2}}{a}{P\nu}\) 
with the property that the composite 
\(\nameddright{S^{2}}{a}{P\nu}{e}{\mathbb{C}P^{\infty}}\) 
induces an isomorphism in degree $2$ cohomology. 
\end{proof} 

Next, the map 
\(\namedright{P\nu}{e}{\mathbb{C}P^{\infty}}\) 
is used to construct a map 
\(\namedright{P_{B}\nu}{}{\mathbb{C}P^{\infty}}\). 
Recall from (\ref{SPnudiag}, \ref{Snu=Slambdaeq}) the projections $q: \partial V\longrightarrow P\nu$ and $q_0: S^{2n-1}\longrightarrow \mathbb{C}P^{n-1}$ are the sphere bundles of the complex line bundles $\lambda_\nu$ and $\eta$ respectively. 
Then as a complex line bundle is classified by its Euler class, we see that the circle bundles $q$ and $q_{0}$ are induced by $e$ and $c$ respectively, that is, there are homotopy fibrations 
\begin{equation} 
  \label{ecfibrations} 
  \nameddright{\partial V}{q}{P\nu}{e}{\mathbb{C}P^{\infty}},\qquad
      \nameddright{S^{2n-1}}{q_{0}}{\mathbb{C}P^{n-1}}{c}{\mathbb{C}P^{\infty}}. 
\end{equation} 
Consider the second pushout in~(\ref{4pos}), the geometric pushout defining $P_{B}\nu$. We have the composite
\[
D\lambda_\nu\stackrel{\pi_\lambda}{\longrightarrow}P\nu\stackrel{e}{\longrightarrow} \mathbb{C}P^\infty.
\] 
Let $\varepsilon: V\rightarrow \mathbb{C}P^\infty$ be the trivial map to the base point.

\begin{lemma}\label{hatelemma} 
There is a map 
\(\namedright{P_{B}\nu}{\widehat{e}}{\mathbb{C}P^{\infty}}\) 
satisfying a homotopy pushout diagram 
\[
\begin{gathered}
\xymatrix{ 
   \partial V \ar[r]^{\iota_\lambda}   \ar[d]^{\iota_\nu}              &  D\lambda_\nu \ar[d]^{\jmath_\nu} \ar@/^/[ddr]^{e\circ\pi_\lambda} &  \\ 
   V \ar[r]^{\jmath_\lambda} \ar@/_/[drr]_{\varepsilon} &  P_B\nu \ar@{.>}[dr]^(0.4){\widehat{e}} & \\ 
   & & \mathbb{C}P^\infty. 
   }  
\end{gathered}
\] 
Consequently, there exists a class $\widehat{e}\in H^2(P_B\nu)$ such that $\jmath_\nu^\ast(\widehat{e})=\pi_\lambda^\ast(e)$ and $\jmath_\lambda^\ast(\widehat{e})=0$. 
\end{lemma}

\begin{proof}
A geometric pushout is always a homotopy pushout, so the geometric pushout 
defining $P_{B}\nu$ is a homotopy pushout. Since $\varepsilon$ is the trivial map, 
so is $\varepsilon\circ\iota_{\nu}$. On the other hand,  
by~(\ref{qfactoreq}), $q=\pi_{\lambda}\circ\iota_{\lambda}$, so the composite 
$e\circ\pi_{\lambda}\circ\iota_{\lambda}=e\circ q$ is null homotopic by~(\ref{ecfibrations}). Therefore 
$\varepsilon\circ\iota_{\nu}\simeq e\circ\pi_{\lambda}\circ\iota_{\lambda}$, so a pushout map 
\(\namedright{P_{B}\nu}{\widehat{e}}{\mathbb{C}P^{\infty}}\) 
exists that makes the entire diagram in the statement of the lemma homotopy commute.
\end{proof}

\begin{lemma} 
   \label{hatesphere} 
   There is a map 
   \(\namedright{S^{2}}{}{P_{B}\nu}\) 
   with the property that the composite 
   \(\nameddright{S^{2}}{}{P_{B}\nu}{\widehat{e}}{\mathbb{C}P^{\infty}}\) 
   induces an isomorphism in degree $2$ cohomology.  
\end{lemma} 

\begin{proof} 
By Lemma~\ref{esphere}, there is a map 
\(\namedright{S^{2}}{}{P\nu}\) 
with the property that the composite 
\(\nameddright{S^{2}}{}{P\nu}{e}{\mathbb{C}P^{\infty}}\) 
induces an isomorphism in degree $2$ cohomology. By Lemma~\ref{pilambdaequiv}, there is a 
homotopy equivalence 
\(\namedright{D\lambda_{\nu}}{\pi_{\lambda}}{P\nu}\). 
Therefore there is a map 
\(\namedright{S^{2}}{}{D\lambda_{\nu}}\) 
with the property that the composite 
\(\llnameddright{S^{2}}{}{D\lambda_{\nu}}{e\circ\pi_{\lambda}}{\mathbb{C}P^{\infty}}\) 
induces an isomorphism in degree $2$ cohomology. By Lemma~\ref{hatelemma}, 
$e\circ\pi_{\lambda}\simeq\widehat{e}\circ\jmath_{\nu}$. Therefore the composite 
\(a\colon\nameddright{S^{2}}{}{D\lambda_{\nu}}{\jmath_{\nu}}{P_{B}\nu}\)
has the property that the composite 
\(\nameddright{S^{2}}{a}{P_{B}\nu}{\widehat{e}}{\mathbb{C}P^{\infty}}\) 
induces an isomorphism in degree $2$ cohomology.
\end{proof}

The class $\widehat{e}\in H^2(P_B\nu)$ in Lemma \ref{hatelemma} determines a principal $S^1$-bundle over $P_B\nu$ 
\begin{equation}\label{S1hatPBnueq}
S^1\stackrel{\widehat{j}}{\longrightarrow} \widehat{P}_B\nu \stackrel{\widehat{s}}{\longrightarrow} P_B\nu 
\end{equation}
that defines the manifold $\widehat{P}_B\nu$ and the maps $\widehat{j}$ and $\widehat{s}$. Notice that ${\rm dim}(\widehat{P}_B\nu)={\rm dim}(P_B\nu)+1=l+2n+1=m+1$. We will need a technical lemma 
showing that the restriction of this circle bundle determined by pulling back with 
\(\namedright{D\lambda_{\nu}}{\jmath_{\nu}}{P_{B}\nu}\) 
is fibrewise isomorphic to $\partial V\times D^{2}$ over $B$. 

\begin{lemma}\label{hatPBnulemma}
There is a commutative diagram of geometric pushouts
\[
\begin{gathered}
\label{hatPBnupushoutdiag2}
\spreaddiagramcolumns{-1pc}\spreaddiagramrows{-1pc} 
\xymatrix{
&&&&S^1 \ar[ddlll] \ar[ddrrr]\\
&~&&&&&&~\\
&~&&&&&&~\\
S^{2n-1}\times S^1 \ar[rr]^{{\rm id}\times \iota}   \ar[dd]^{(\iota\times {\rm id})\circ \phi_x} &&  S^{2n-1} \times D^2 \ar[dd]^{} & && &
\partial V\times S^1 \ar[rr]^{{\rm id} \times \iota} \ar[dd]^{(\iota_\nu \times {\rm id})\circ \Phi} && 
\partial V \times D^2 \ar[dd]\\
&&&\ar[rr] &&&&&& \ar[dddr] \\
D^{2n}\times S^1 \ar[rr]^{}                        &&   S^{2n+1}, &&&&
V \times S^1 \ar[rr] &&
\widehat{P}_B\nu\\
&\ar[dd]&&&&&&\ar[dd]\\
& ~&&&&&&~ &&& B\\
& ~&&&&&&~\\
S^{2n-1}  \ar[rr]^{\iota_\eta}        \ar[dd]^{\iota}      &&   D\eta  \ar[dd]^{}   & &&&
\partial V \ar[rr]^{\iota_\lambda}   \ar[dd]^{\iota_\nu} &&  D\lambda_\nu \ar[dd]^{\jmath_\nu}  \\
&&&\ar[rr] &&&&&& \ar[uuur] \\
D^{2n} \ar[rr]^{}      && \mathbb{C}P^{n},    &&&&
V \ar[rr]^{\jmath_\lambda}                                         &&   P_B\nu
}
\end{gathered}
\] 
where the horizontal morphisms of pushouts are fibred over $B$, the vertical morphisms of pushouts are
 circle bundles, $\Phi$ is a fibrewise self-diffeomorphism of $\partial V\times S^{1}$, for any $x\in B$ 
 the map $\phi_{x}$ is the restriction of $\Phi$ to $S^{2n-1}\times S^{1}$, and $\iota$ and $\iota_{\nu}$ 
 are the inclusions of the boundaries. 
 \end{lemma} 

\begin{proof} 
By Lemma \ref{kappanulemma}, the geometric pushout for $P_B\nu$ in~(\ref{PBnudiag}) is fibred 
over $B$ with the pushout of~$\mathbb{C}P^n$ in~(\ref{PBnudiag1}) as fibre. This 
accounts for the lower horizontal row in the diagram asserted to exist by the lemma. 

Next, compose all maps in this lower row with 
\(\namedright{P_{B}\nu}{\widehat{e}}{\mathbb{C}P^{\infty}}\) 
and take fibres. First consider the right column in the diagram asserted to exist by the lemma. 
The map $\widehat{e}$ induces the circle bundle 
$S^1\stackrel{\widehat{j}}{\longrightarrow} \widehat{P}_B\nu \stackrel{\widehat{s}}{\longrightarrow} P_B\nu$ 
in~(\ref{S1hatPBnueq}). By Lemma \ref{hatelemma}, the composite 
\(\nameddright{V}{\jmath_{\lambda}}{P_{B}\nu}{\widehat{e}}{\mathbb{C}P^{\infty}}\) 
is homotopic to the trivial map to the basepoint and the composite 
\(\nameddright{D\lambda_{\nu}}{\jmath_{\nu}}{P_{B}\nu}{\widehat{e}}{\mathbb{C}P^{\infty}}\) 
is homotopic to $e\circ\pi_{\lambda}$. Therefore the pullback of 
$\widehat{P}_B\nu \stackrel{\widehat{s}}{\longrightarrow} P_B\nu$ 
with 
\(\namedright{V}{\jmath_{\lambda}}{P_{B}\nu}\) 
is isomorphic to $V\times S^1$ while its pullback with 
\(\namedright{D\lambda_{\nu}}{\jmath_{\nu}}{P_{B}\nu}\)
is $\pi_\lambda^{-1}(S\lambda_\nu)$. 
We claim that $\pi_\lambda^{-1}(S\lambda_\nu)$ is fibrewise isomorphic to $\partial V\times D^2$, meaning there is a fibration diagram 
\begin{equation} 
  \label{Bfibrewise} 
  \diagram 
      S^{2n-1}\times D^{2}\rto\dto^{\cong} & \partial V\times D^{2}\rto\dto^{\cong} & B\ddouble \\ 
      \pi_{\eta}^{-1}(S\eta)\rto & \pi_{\lambda}^{-1}(S\lambda_{\nu})\rto & B. 
  \enddiagram 
\end{equation} 
If this were to exist then we would have established that the right column in the diagram asserted by 
the lemma is a morphism of pushouts over $\mathbb{C}P^{\infty}$, or equivalently, a circle bundle of 
pushouts. 

To establish~(\ref{Bfibrewise}), consider the pullback squares of fibre bundles
\[
\begin{gathered}
\xymatrix{
\pi_\eta^{-1}(S\eta)\cong q_0^{-1}(D\eta) \ar[r] \ar[d]  &  D\eta \ar[d]^{\pi_\eta}     & 
\pi_\lambda^{-1}(S\lambda_\nu)\cong q^{-1}(D\lambda_\nu) \ar[r] \ar[d]  &  D\lambda_\nu \ar[d]^{\pi_\lambda}    \\
S\eta \cong S^{2n-1} \ar[r]^<<<<<{q_0} & \mathbb{C}P^{n-1},                    &
S\lambda_\nu\cong\partial V \ar[r]^<<<<<<{q} & P\nu,
 }
 \end{gathered}
\]
where the right pullback is fibred over $B$ with the left pullback as fibre. This implies that $\pi_\lambda^{-1}(S\lambda_\nu)\cong q^{-1}(D\lambda_\nu)$ is a fibrewise isomorphism and restricts to $\pi_\eta^{-1}(S\eta)\cong q_0^{-1}(D\eta)$ on each fibre. 
Recall that the disk bundle projections $D\lambda_\nu\stackrel{\pi_\lambda}{\longrightarrow} P\nu$ and $D\eta\stackrel{\pi_\eta}{\longrightarrow} \mathbb{C}P^{n-1}$ are classified by 
\(\namedright{P\nu}{e}{\mathbb{C}P^{\infty}}\) 
and 
\(\namedright{\mathbb{C}P^{n-1}}{c}{\mathbb{C}P^{\infty}}\) 
respectively, and $e$ restricts to $c$ on each fibre by \eqref{ecdgrm}. Additionally, the composites $e\circ q$ and $c\circ q_0$ are null homotopic by~(\ref{ecfibrations}). Therefore, the pullback bundles $q^{-1}(D\lambda_\nu)$ and $q_0^{-1}(D\eta)$ are trivial, and there is a bundle isomorphism $q^{-1}(D\lambda_\nu)\cong\partial V\times D^2$ over $\partial V$, with restriction $q_0^{-1}(D\eta)\cong S^{2n-1}\times D^2$ over each fibre $S^{2n-1}$. Hence there is a fibrewise isomorphism $\pi_\lambda^{-1}(S\lambda_\nu)\cong\partial V\times D^2$, which restricts to $\pi_\eta^{-1}(S\eta)\cong S^{2n-1}\times D^2$. 

Finally, the upper row and left column in the diagram in the statement of the lemma 
follow by pulling back the lower row and right column.   
\end{proof}

We now turn to constructing a canonical circle bundle over the geometric blow up $\widetilde{N}=N\mathop{\#}\limits_{B}P_B\nu$. Consider the fourth pushout in~(\ref{4pos}), the geometric pushout 
defining $\widetilde{N}$. Let 
\(\widetilde{\varepsilon}\colon\namedright{N_{c}}{}{\mathbb{C}P^{\infty}}\) 
be the trivial map to the base point. Arguing just as for Lemmas~\ref{hatelemma} and~\ref{hatesphere} 
we obtain the following.

\begin{lemma} 
\label{tildeelemma} 
There is a map 
\(\namedright{\widetilde{N}}{\widetilde{e}}{\mathbb{C}P^{\infty}}\) 
satisfying a homotopy pushout diagram 
\[
\begin{gathered}
\xymatrix{ 
   \partial V \ar[r]^{\iota_\lambda}   \ar[d]^{\iota_c}              &  D\lambda_\nu \ar[d]^{r_{j}} \ar@/^/[ddr]^{e\circ\pi_\lambda} &  \\ 
   N_{c} \ar[r]^{r} \ar@/_/[drr]_{\widetilde{\varepsilon}} &  \widetilde{N} \ar@{.>}[dr]^(0.4){\widetilde{e}} & \\ 
   & & \mathbb{C}P^\infty. 
   }  
\end{gathered}
\] 
Consequently, there exists a class $\widetilde{e}\in H^2(\widetilde{N})$ such that $r_{j}^{\ast}(\widetilde{e})=\pi_\lambda^\ast(e)$ and $r^{\ast}(\widetilde{e})=0$. 
Further, there is a map 
\(\namedright{S^{2}}{}{\widetilde{N}}\) 
with the property that the composite 
\(\nameddright{S^{2}}{}{\widetilde{N}}{\widetilde{e}}{\mathbb{C}P^{\infty}}\) 
induces an isomorphism in degree $2$ cohomology. ~$\qqed$
\end{lemma}

The class $\widetilde{e}\in H^{2}(\widetilde{N})$ in Lemma \ref{tildeelemma} determines a principal $S^1$-bundle 
\begin{equation}\label{S1tildeEeq}
S^1\stackrel{\widetilde{j}}{\longrightarrow} \widetilde{E} \stackrel{\widetilde{s}}{\longrightarrow} \widetilde{N} 
\end{equation}
that defines the $(m+1)$-dimensional manifold $\widetilde{E}$ and the maps $\widetilde{j}$ and $\widetilde{s}$.
We call (\ref{S1tildeEeq}) the {\it canonical circle bundle} of the geometric blow up $\widetilde{N}$.

The following proposition is the main result of this section. Recall from 
Lemma~\ref{blowup=sumlemma} that the geometric blow up $\widetilde{N}$ can be identified with the 
fibrewise connected sum $N\conn_{B} P_{B}\nu$. 

\begin{proposition}\label{tildeEprop}
The total manifold $\widetilde{E}$ of the canonical circle bundle (\ref{S1tildeEeq}) is the effect of a fibrewise $1$-surgery on the manifold $N\times S^{1}$ determined by a fibrewise framed $1$-embedding $g: V\times S^{1}\hookrightarrow  N\times S^1$ over $B$:
\[
\widetilde{E}\cong ((N\times S^1)\backslash g(\accentset{\circ}{V}\times  S^{1}))\cup_{\partial V\times S^{1}} (\partial V\times D^{2}).
\]
Moreover, there is a geometric pushout
\begin{equation}
\begin{gathered}
\label{tildeEpushoutdiag}
\xymatrix{
\partial V\times S^1 \ar[r]^{{\rm id} \times \iota} \ar[d]^{(\iota_c\times {\rm id})\circ \phi} & 
\partial V \times D^2 \ar[d]\\
N_c\times S^1 \ar[r] &
\widetilde{E}, 
}
\end{gathered}
\end{equation} 
where $\phi$ is the restriction to $\partial V\times S^{1}$ of a self-diffeomorphism $\phi'$ 
of $V\times S^{1}$ given by $\phi'(v, t)=(\tau(t)v, t)$ for some map $\tau: S^1\longrightarrow \mathcal{G}^{+}(\nu)$. 
\end{proposition}
\begin{proof}
By Lemmas \ref{blowup=sumlemma} and \ref{tildeelemma}, the circle bundle (\ref{S1tildeEeq}) over 
\[
\widetilde{N}\cong N\mathop{\#}\limits_{B}P_B\nu \cong  (N \backslash i(\accentset{\circ}{V}))\cup_{\partial V}(\partial V\times [0,1])\cup_{\partial V}(P_B\nu \backslash \jmath_{\lambda}(\accentset{\circ}{V}))
\]
is determined by the class $\widetilde{e}\in H^2(\widetilde{N})$, whose restriction to $N_c=N \backslash i(\accentset{\circ}{V})$ is trivial and whose restriction to $P_B\nu \backslash \jmath_{\lambda}(\accentset{\circ}{V})=D\lambda_{\nu}$ is 
determined by $\pi_{\lambda}^{\ast}(e)\in H^2(D\lambda_{\nu})$. By Lemma~\ref{hatelemma}, 
$\pi_{\lambda}^{\ast}(e)=\jmath_{\nu}^{\ast}(\widehat{e})$, so the circle bundle determined by 
$\pi_{\lambda}^{\ast}(e)$ is given by restricting the circle bundle 
\(\namedright{\widehat{P}_{B}\nu}{\widehat{s}}{P_{B}\nu}\) in (\ref{S1hatPBnueq}) 
determined by $\widehat{e}$ with the map 
\(\namedright{D\lambda_{\nu}}{\jmath_{\nu}}{P_{B}\nu}\). 
By the right column in the diagram in the statement of Lemma~\ref{hatPBnulemma}, this circle 
bundle is fibrewise isomorphic to $\partial V\times D^2$ over $B$. 

This implies that, up to isomorphism, the circle bundle (\ref{S1tildeEeq}) can be obtained by first considering a trivial bundle $N\times S^1$ over $N$ and the canonical circle bundle (\ref{S1hatPBnueq}) over $P_B\nu$, then removing 
two copies of $V\times S^1$ from $N\times S^1$ and $P_B\nu$ respectively, and finally gluing them together along the boundary. Equivalently, to construct the total manifold $\widetilde{E}$ of the circle bundle (\ref{S1tildeEeq}) up to isomorphism, we can first remove a suitable fibrewise framed $1$-embedding $g: V\times S^{1}\hookrightarrow  N\times S^1$ determined by the two copies of $V\times S^1$, and then glue it with $\partial V\times D^{2}$ through $g$ along $\partial V\times S^1$. By Definition \ref{semi-surdef}, this means that $\widetilde{E}$ is a fibrewise $1$-surgery on $N\times S^1$, that is, 
\begin{equation}\label{Etildeiso} 
\widetilde{E}\cong ((N\times S^1)\backslash g(\accentset{\circ}{V}\times  S^{1}))\cup_{\partial V\times S^{1}} (\partial V\times D^{2}).
\end{equation} 

Moreover, notice there is a standard fibrewise framed $1$-embedding $g_0=i\times {\rm id}: V\times S^{1}\hookrightarrow  N\times S^1$, where $i$ is the inclusion. By Lemma \ref{fib-frame-lemma} (1), up to an isotopy relative to the core manifold embedding $B\times S^{1}\hookrightarrow N\times S^1$, the two framed embeddings~$g$ and $g_0$ differ by a fibrewise diffeomorphism
\[
\phi': V\times S^{1} \longrightarrow V\times S^{1}
\]
over $B\times S^{1}$ with the property that $\phi'(v, t)= (\tau(t) v, t)$ for some map $\tau: S^1 \rightarrow \mathcal{G}^{+}(\nu)$. 
In particular, $g$ is isotopic to $(i\times {\rm id})\circ \phi'$ relative to $B\times S^{1}\hookrightarrow N\times S^1$. Hence, by Lemma \ref{fib-frame-lemma} (2) the effects of fibrewise $1$-surgeries along $g$ and $(i\times {\rm id})\circ \phi'$ are diffeomorphic. Combining with~(\ref{Etildeiso}), this implies that 
\[
\begin{split}
\widetilde{E}
&\cong ((N\times S^1)\backslash g(\accentset{\circ}{V}\times  S^{1}))\cup_{\partial V\times S^1} (\partial V\times D^{2})\\
&\cong ((N\times S^1)\backslash ((i\times {\rm id})\circ \phi')(\accentset{\circ}{V}\times  S^{1}))\cup_{\partial V\times S^{1}} (\partial V\times D^{2}).
\end{split}
\] 
Since the restrictions of $\phi'$ and $i\times {\rm id}$ on the boundaries are $\phi$ and $\iota_c\times {\rm id}$ respectively, the last diffeomorphism is equivalent to the asserted geometric pushout in~(\ref{tildeEpushoutdiag}).
\end{proof}

\begin{example}[Circle bundle over a blow up at a point]\label{S1blowuppointex}
Suppose $B=\ast$ is a point and continue Example \ref{blowuppointex}. In this case $\widetilde{e}\in H^2(N\#\mathbb{C}P^n)$ corresponds to a generator $c\in H^2(\mathbb{C}P^n)$ for the right side 
of the connected sum. Diagram (\ref{tildeEpushoutdiag}) of Proposition \ref{tildeEprop} becomes
\[
\begin{gathered}
\label{sumextildeEpushoutdiag}
\xymatrix{
S^{2n-1}\times S^1 \ar[r]^{{\rm id} \times \iota} \ar[d]^{(i\times {\rm id})\circ \phi} & 
S^{2n-1} \times D^2 \ar[d]\\
N_0\times S^1 \ar[r] &
\widetilde{E}.
}
\end{gathered}
\] 
Thus the manifold $\widetilde{E}$ is obtained from $N\times S^1$ by a surgery along $\ast\times S^{1}\cong S^1$ with framing determined by $\phi$. This was studied by the authors~\cite{HT2} from the perspective of homotopy theory and by Duan~\cite{D1} from the perspective of geometric topology. 
\end{example}


\section{The local homotopy theory of a blow up}
\label{sec: homotopy} 
In this section we extend the method in \cite[Section 2]{HT2} to study the local homotopy 
theory of blow ups and prove Theorems \ref{loopblowupthmintro}, \ref{loopblowupthmintrorational} 
and~\ref{dichotomythmintro}. 

We start with an initial decomposition. By~(\ref{S1tildeEeq}), if $\widetilde{N}$ is the geometric blow up of $N$ along an $l$-dimensional submanifold $B$ with normal bundle $\nu$, then there is a canonical $S^1$-bundle $S^1\stackrel{\widetilde{j}}{\longrightarrow} \widetilde{E} \stackrel{\widetilde{s}}{\longrightarrow} \widetilde{N}$. 
\begin{lemma}\label{canbundletriv}
There is a homotopy equivalence
\[
\Omega \widetilde{N}\simeq S^1\times \Omega \widetilde{E}.
\]
\end{lemma}

\begin{proof}
By definition, the canonical $S^1$-bundle 
$S^1\stackrel{\widetilde{j}}{\longrightarrow} \widetilde{E} \stackrel{\widetilde{s}}{\longrightarrow} \widetilde{N}$ 
is induced by the map 
\(\namedright{\widetilde{N}}{\widetilde{e}}{\mathbb{C}P^{\infty}}\).  
By Lemma~\ref{tildeelemma}, there is a map 
\(\namedright{S^{2}}{}{\widetilde{N}}\) 
with the property that the composite 
\(\nameddright{S^{2}}{}{\widetilde{N}}{\widetilde{e}}{\mathbb{C}P^{\infty}}\) 
induces an isomorphism in degree $2$ cohomology. Therefore, after looping, if 
\(\namedright{S^{1}}{E}{\Omega S^{2}}\) 
is the suspension (equivalent to the inclusion of the bottom cell), then the composite 
\(\namedddright{S^{1}}{E}{\Omega S^{2}}{}{\Omega\widetilde{N}}{\Omega\widetilde{e}}{S^{1}}\) 
induces an isomorphism in cohomology. The composite therefore induces an isomorphism 
in homology, implying by the Hurewicz isomorphism that it is homotopic to the identity map up to sign. 
Hence the homotopy fibration 
\(\nameddright{\Omega\widetilde{E}}{\Omega\widetilde{s}}{\Omega\widetilde{N}}{\Omega\widetilde{e}}{S^{1}}\) 
induced by~(\ref{S1tildeEeq}) splits to give the homotopy equivalence asserted by the lemma. 
\end{proof}

By Lemma \ref{canbundletriv}, the homotopy theory of $\widetilde{N}$ can be studied via  
that of $\widetilde{E}$. Doing so will involve localization in order to establish 
conditions that allow for the inclusion 
\(N_{c}\hookrightarrow\widetilde{E}\)  
to have a left homotopy inverse. This will be a consequence of Proposition~\ref{hextlemma}, 
which requires several ingredients. 

First, we need to ensure that localization makes sense. This will be the case if all 
the manifolds involved: $N$, $B$, $N_c$, $\partial V$ and $\widetilde{N}$, are all connected 
nilpotent spaces. On the one hand, this can be made a blanket assumption. 
On the other hand, a simple criterion is given by the next lemma, which assumes only that 
the input spaces $N$ and $B$ are simply-connected and $n\geq 2$. Recall that any 
simply-connected spaces is nilpotent.

\begin{lemma}\label{nillemma}
If $N$ and $B$ are simply connected and $n\geq 2$, then the manifolds $N$, $B$, $\partial V$, $N_c$ and~$\widetilde{N}$ are all simply-connected. 
\end{lemma}

\begin{proof}
By hypothesis, $N$ and $B$ are simply-connected. The long exact sequence of homotopy groups induced by the sphere bundle $S^{2n-1}\stackrel{}{\longrightarrow} \partial V\stackrel{}{\longrightarrow} B$ implies that $\partial V$ is simply connected. 
As $V\simeq B$ and $N$ are simply connected, the Van-Kampen theorem can be applied to $N\cong V\cup_{\partial V} N_c$ to show that $N_c$ is simply connected. A similar argument for the decomposition $\widetilde{N}\cong P\nu\cup_{\partial V} N_c$ implies that $\pi_1 (\widetilde{N})\cong \pi_1(N)=0$ \cite[Theorem 8.22]{FOT}, and thus $\widetilde{N}$ is a simply connected.
\end{proof} 

The localization needed will be determined by the image of the classical 
$J$-homomorphism, which is introduced next.
Let $\Omega^{2n} X$ be the $2n$-fold iterated loop space of a space $X$. Consider the composite
\[
SO(2n)\stackrel{\chi}{\longrightarrow} {\rm Map}(S^{2n-1}, S^{2n-1})\stackrel{E}{\longrightarrow} {\rm Map}^\ast(S^{2n}, S^{2n})=\Omega^{2n} S^{2n},
\]
where $\chi$ is induced by the standard action of $SO(2n)$ on $S^{2n-1}$ and $E$ is the suspension that maps the free mapping space ${\rm Map}(S^{2n-1}, S^{2n-1})$ to the based mapping space ${\rm Map}^\ast(S^{2n}, S^{2n})$. 
The $J$-homomorphism is defined as the morphism induced by $E\circ \chi$ on homotopy groups
\[
J: \pi_{k-1} (SO(2n))\rightarrow \pi_{k-1}(\Omega^{2n} S^{2n})\cong \pi_{2n+k-1}(S^{2n}).
\]

In the stable range, that is, when $2n\geq k+1$, famous work of Adams \cite{A} and Quillen \cite{Q} shows that the image of $J$ is
\begin{equation}\label{imJeq}
{\rm Im}\, J\cong \left\{\begin{array}{cc}
0 & k\equiv 3, 5, 6, 7 \ {\rm mod} \ 8, \\
\mathbb{Z}/2 & k\equiv 1, 2\ {\rm mod} \ 8, k\neq 1, \\
\mathbb{Z}/d_s& k=4s,
\end{array}\right.
\end{equation}
where $d_s$ is the denominator of $B_s/4s$ and $B_s$ is the $s$-th Bernoulli number defined by 
\[\frac{z}{e^z-1}=1-\frac{1}{2}z-\sum\limits_{s\geq 1} (-1)^sB_s \frac{z^{2s}}{(2s)!}.\] 
For each $k\geq 2$, let $\mathcal{P}_{k}$ be the set of prime numbers such that
\[\label{pkdefeq}
\mathcal{P}_{k}= \left\{\begin{array}{cc}
\emptyset & k\equiv 3, 5, 6, 7 \ {\rm mod} \ 8, \\
\{2\} & k\equiv 1, 2\ {\rm mod} \ 8, k\neq 1, \\
\{p~|~p~{\rm divides}~d_s\}& k=4s,
\end{array}\right.
\]
and for each $l\geq k$, let $\mathcal{Q}_{k,l}$ be the set of prime numbers defined by
\[\label{qmkdefeq}
\mathcal{Q}_{k,l}=\{2\}\cup \mathcal{P}_{k}\cup \cdots \cup \mathcal{P}_{l}.
\] 

\begin{lemma}\label{Bslemma}
For each $l\geq k$,
\[\label{qmkdefeq2}
\mathcal{Q}_{k,l}=\{ {\rm prime}~p~|~(p-1)~{\rm divides}~2s,~k\leq 4s \leq l\}.
\] 
\end{lemma}
\begin{proof} 
Let $p$ be an odd prime. By \cite[Theorem B.4]{MS}, $p$ divides the denominator $d_{s}$ of $B_s/4s$ 
if and only if it divides the denominator of $B_s$, while by \cite[Theorem B.3]{MS} the latter holds 
if and only if $p-1$ divides $2s$. Thus 
$\{{\rm odd}~p\mid p~{\rm divides}~d_{s}\}=\{{\rm odd}~p\mid (p-1)~{\rm divides}~2s\}$. Therefore 
$\mathcal{Q}_{k,l}\backslash\{2\}=\{ {\rm odd\ prime}~p~|~(p-1)~{\rm divides}~2s,~k\leq 4s \leq l\}$. Finally, 
by definition $2\in\mathcal{Q}_{k.l}$ while $2-1$ divides $2s$, so 
$\mathcal{Q}_{k,l}=\{ {\rm prime}~p~|~(p-1)~{\rm divides}~2s,~k\leq 4s \leq l\}$. 
\end{proof}

\noindent 
\textbf{A key proposition}.
Two further ingredients are needed to prove Proposition \ref{hextlemma}, which is the key to constructing a left homotopy inverse of the inclusion $N_c\stackrel{}{\hookrightarrow}\widetilde{E}$. The first is an auxilliary lemma. 

\begin{lemma}\label{QimJ=0lemma}
Let $Y$ be a $(k-1)$-connected $l$-dimensional $CW$-complex such that $l\leq 2n-4$ and $k\geq 1$. Let~$p$ be a prime such that $p>\frac{1}{2}(l-k)+1$, $p \not\in \mathcal{Q}_{k+2,l+2}$ and $H_\ast(Y;\mathbb{Z})$ is $p$-torsion free. Then the composition
\[
\Sigma Y\stackrel{f}{\longrightarrow} SO(2n) \stackrel{\chi}{\longrightarrow} {\rm Map}(S^{2n-1}, S^{2n-1})
\]
is $p$-locally null homotopic for any $f$. 
\end{lemma}

\begin{proof}
Localize spaces and maps at $p$. By assumption, the cells of $\Sigma Y$ are concentrated in dimensions $k+1$ to $l+1$. 
By \cite[Lemma~5.1]{HT1}, since $p>\frac{1}{2}(l-k)+1$ and $H_\ast(Y;\mathbb{Z})$ is $p$-torsion free, $\Sigma Y$ is homotopy equivalent to a wedge of spheres. Therefore the composite 
\[
\Sigma Y\stackrel{f}{\longrightarrow} SO(2n) \stackrel{\chi}{\longrightarrow} {\rm Map}(S^{2n-1}, S^{2n-1})\stackrel{E}{\longrightarrow} \Omega^{2n}S^{2n}
\]
is determined by its restriction to each sphere summand of $\Sigma Y$. Each such restriction 
factors through the $J$-homomorphism $E\circ\chi$, and so is null homotopic since the assumption 
that $p\notin\mathcal{Q}_{k+2,l+2}$ implies that the image of the $J$-homorphism is trivial. Hence 
$E\circ\chi\circ f$ is null homotopic.

On the other hand, since $S^{2n-1}$ is $(2n-2)$-connected, the long exact sequence of the homotopy groups of the canonical evaluation fibration 
\[
\Omega^{2n-1}S^{2n-1}\stackrel{}{\longrightarrow}{\rm Map}(S^{2n-1}, S^{2n-1})\stackrel{{\rm ev}}{\longrightarrow} S^{2n-1},
\]
implies that $\pi_{i}({\rm Map}(S^{2n-1}, S^{2n-1}))\cong\pi_{i}(\Omega^{2n-1}S^{2n-1})$ for any $i\leq 2n-3$.
Meanwhile, the Freudenthal suspension theorem implies that $E: \pi_{i}(\Omega^{2n-1}S^{2n-1})\longrightarrow \pi_{i}(\Omega^{2n}S^{2n})$ is an isomorphism for any $i\leq 2n-3$. Since $l+1\leq 2n-3$ by assumption, combining these two statements implies that $[\Sigma Y, {\rm Map}(S^{2n-1}, S^{2n-1})]\stackrel{E_{\ast}}{\longrightarrow} [\Sigma Y, \Omega^{2n}S^{2n}]$ is an isomorphism, where $[-,-]$ denotes the set of based homotopy classes of based maps.
Hence $E\circ \chi\circ f$ being null homotopic implies that $\chi\circ f$ is as well. This proves the lemma.
\end{proof}

The other ingredient needed for Proposition~\ref{hextlemma} is some homotopy theoretic information about the gauge group $\mathcal{G}^{+}(\nu)$ and an associated topological monoid of fibrewise self-homotopy equivalences. 
As general notation, suppose that $G$ is a topological group, $X$ is a pointed space, 
and $P\longrightarrow X$ is a principal $G$-bundle classified by a map $f\colon X\longrightarrow BG$, 
where $BG$ is the classifying space of $G$. Then there is an evaluation fibration 
\begin{equation}\label{mapevfib} 
\nameddright{{\rm Map}_{f}^{\ast}(X,BG)}{}{{\rm Map}_{f}(X,BG)}{ev}{BG} 
\end{equation} 
where ${\rm Map}_{f}(X,BG)$ is the component of the space of continuous maps ${\rm Map}(X,BG)$ that 
contains the map $f$, ${\rm Map}_{f}^{\ast}(X,BG)$ is the component of the space of pointed 
continuous maps ${\rm Map}^\ast(X,BG)$ that contains $f$, and $ev$ evaluates a map at the basepoint of $X$. 

Consider the bundle~$\nu$ over $B$ with sphere bundle $S^{2n-1}\stackrel{i_\nu}{\longrightarrow}\partial V\stackrel{s_\nu}{\longrightarrow} B$ as in (\ref{SPnudiag}). The bundle $\nu$ is 
classified by a map 
\(\namedright{B}{}{BSO(2n)}\) 
that will also (ambiguously) be denoted by $\nu$. 
On the one hand, by \cite{Got} or \cite{AB} there is a homotopy equivalence
\begin{equation}\label{ABguage}
B\mathcal{G}^{+}(\nu)\simeq {\rm Map}_\nu(B, BSO(2n))
\end{equation}
where $B\mathcal{G}^{+}(\nu)$ is the classifying space of the gauge group $\mathcal{G}^{+}(\nu)$. 
Under this homotopy equivalence, the evaluation fibration~(\ref{mapevfib}) 
with $X=B$ and $G=SO(2n)$ becomes
\begin{equation}\label{guagefibeq}
{\rm Map}_\nu^\ast(B, BSO(2n)) \stackrel{g_\nu}{\longrightarrow}  B\mathcal{G}^{+}(\nu)\stackrel{{\rm ev}}{\longrightarrow}  BSO(2n). 
\end{equation} 
On the other hand, let ${\rm aut}_B(\partial V)$ be the topological monoid of fibrewise self-homotopy equivalences of $\partial V\stackrel{s_\nu}{\longrightarrow} B$ as a spherical fibration. 
Dold and Lashof~\cite{DL} showed that its classifying space $B{\rm aut}_B(\partial V)$ classifies the fibrewise homotopy equivalences classes of spherical fibrations over~$B$.
By \cite[Theorem 3.3]{BHMP} or \cite[Proposition 2.1]{BS}, there is a homotopy equivalence
\begin{equation}\label{ABaut}
B{\rm aut}_B(\partial V)\simeq {\rm Map}_\nu(B, B{\rm aut}(S^{2n-1})). 
\end{equation}
Under this homotopy equivalence, there is an evaluation fibration 
\begin{equation}\label{autfibeq}
{\rm Map}^\ast_\nu(B, B{\rm aut}(S^{2n-1})) \stackrel{g_s}{\longrightarrow}  B{\rm aut}_B(\partial V)\stackrel{{\rm ev}}{\longrightarrow}  B{\rm aut}(S^{2n-1}).
\end{equation}

We will relate the evaluation fibrations~(\ref{guagefibeq}) and~(\ref{autfibeq}) after a preliminary lemma.  
\begin{lemma} 
\label{G+aut} 
There are group homomorphisms
$\chi_\nu\colon\mathcal{G}^{+}(\nu)\stackrel{}{\longrightarrow} {\rm aut}_B(\partial V)$ 
and $\chi'\colon\namedright{SO(2n)}{}{{\rm aut}(S^{2n-1})}$ that satisfy a commutative diagram 
\[ 
\xymatrix{ 
\mathcal{G}^{+}(\nu) \ar[r]^{{\Omega {\rm ev}}} \ar[d]^{\chi_\nu} & SO(2n) \ar[d]^{\chi^\prime} \\
{\rm aut}_B(\partial V)\ar[r]^{{\Omega {\rm ev}}} & {\rm aut}(S^{2n-1}).  
   }  
\]
\end{lemma} 

\begin{proof} 
Define 
$\chi_\nu: \mathcal{G}^{+}(\nu)\stackrel{}{\longrightarrow} {\rm aut}_B(\partial V)$ 
as follows. Regard the $n$-dimensional complex vector bundle~$\nu$ over $B$ as a 
$2n$-dimensional real oriented vector bundle. This has an associated principal $SO(2n)$-bundle 
\(\namedright{P}{}{B}\). 
An element $\mathfrak{g}\in\mathcal{G}^{+}(\nu)$ is an $SO(2n)$-equivariant automorphism 
of $P$. Define an automorphism $\chi_{\nu}(\mathfrak{g})\in {\rm aut}_{B}(\partial V)$ by 
having $\mathfrak{g}$ act fibrewise on the sphere bundle 
\(\nameddright{S^{2n-1}}{}{\partial V}{}{B}\). 
Note that the restriction of $\mathfrak{g}$ to the fibre over the basepoint of $B$ is an 
automorphism $\chi'(\mathfrak{g})$ of $S^{2n-1}$. Thus we obtain the commutative diagram asserted by the lemma. 
Further, since $\chi_{\nu}$ is defined via a fibrewise action, we have 
$\chi_{\nu}(\mathfrak{g}\circ\mathfrak{h})=\chi_{\nu}(\mathfrak{g})\circ\chi_{\nu}(\mathfrak{h})$, 
so $\chi_{\nu}$ is a group homomorphism, and hence so is $\chi'$. 
\end{proof} 

The fact that $\chi_{v}$ and $\chi'$ are group homomorphisms implies that they can be delooped, and thus so can the commutative diagram in the statement of Lemma~\ref{G+aut}. Therefore the evaluation fibrations~(\ref{guagefibeq}) and~(\ref{autfibeq}) imply the following.

\begin{lemma} 
\label{G+autev} 
There is a homotopy commutative diagram of homotopy fibrations  
\[
\xymatrix{ 
 {\rm Map}_\nu^\ast(B, BSO(2n))  \ar[d]^{\chi^\prime_\ast}  \ar[r]^<<<<<{g_\nu} & B\mathcal{G}^{+}(\nu) \ar[r]^{{{\rm ev}}} \ar[d]^{B\chi_\nu} & BSO(2n) \ar[d]^{B\chi^\prime} \\
 {\rm Map}_\nu^\ast(B, B{\rm aut}(S^{2n-1})) \ar[r]^<<<{g_s} & B{\rm aut}_B(\partial V)\ar[r]^{{{\rm ev}}} & B{\rm aut}(S^{2n-1})  
   } 
\] 
where $\chi'_{\ast}$ is induced by taking mapping spaces with respect to $B\chi'$.~$\qqed$ 
\end{lemma}  

Recall from Proposition~\ref{tildeEprop} that there is a geometric pushout 
\[\xymatrix{
\partial V\times S^1 \ar[r]^{{\rm id} \times \iota} \ar[d]^{(\iota_c\times {\rm id})\circ \phi} & 
\partial V \times D^2 \ar[d]\\
N_c\times S^1 \ar[r] &
\widetilde{E} 
}\] 
and a map $\tau: S^1\longrightarrow \mathcal{G}^{+}(\nu)$ such that $\phi(v, t)=(\tau(t)v, t)$ 
for any $(v, t)\in \partial V\times S^1$. The purpose of the following proposition is 
to give conditions that will lead to a pushout map 
\(\namedright{\widetilde{E}}{}{N_{c}}\).

\begin{proposition}\label{hextlemma}
Let $n\geq 2$. Suppose that $B$ is a $(k-1)$-connected $l$-dimensional $CW$-complex such that $l\leq 2n-4$ and $k\geq 1$. Let $p$ be a prime such that $p>\frac{1}{2}(l-k)+1$, $p \not\in \mathcal{Q}_{k+2,l+2}$ and $H_\ast(B;\mathbb{Z})$ is $p$-torsion free. 
Then localized at the prime $p$, there exists a map $s: \partial V\times D^{2}\rightarrow N_c$ such that the diagram
\begin{equation} 
  \label{sdiag} 
  \diagram 
      \partial V\times S^{1}\rto^-{{\rm id}\times \iota}\dto^{(\iota_c\times {\rm id})\circ \phi} &  \partial V \times D^{2}\dto^{s} \\ 
    N_c\times S^{1}\rto^-{\pi_1} & N_c
  \enddiagram 
\end{equation} 
homotopy commutes, where $\pi_{1}$ is the projection onto the first factor.
\end{proposition}
\begin{proof}
By definition of $\phi$, $(\pi_1\circ (\iota_c\times {\rm id})\circ \phi)(v, t)=\iota_c(\tau(t)v)$. Thus the composite $\pi_1\circ (\iota_c\times {\rm id})\circ \phi$ has as adjoint the composite 
\[
S^1\stackrel{\tau}{\longrightarrow} \mathcal{G}^{+}(\nu)\stackrel{\chi_\nu}{\longrightarrow} {\rm aut}_B(\partial V)\stackrel{r}{\hookrightarrow} {\rm Map}(\partial V, \partial V)\stackrel{\iota_{c\ast}}{\longrightarrow} {\rm Map}(\partial V, N_c),
\]
where $r$ is the obvious inclusion and $\iota_{c\ast}$ is induced by the inclusion $\partial V\stackrel{\iota_c}{\hookrightarrow} N_c$.
Therefore, to prove the existence of~(\ref{sdiag}), it is equivalent to show that $\iota_{c\ast}\circ r\circ \chi_\nu\circ \tau$ extends to a map 
\(s\colon\namedright{D^{2}}{}{{\rm Map}(\partial V, N_c)}\), which in turn is equivalent to showing that 
$\iota_{c\ast}\circ r\circ \chi_\nu\circ \tau$ is null homotopic. In fact, we claim that $\chi_\nu\circ \tau$ is null homotopic after localization at the prime $p$, and this will prove the lemma.

Localize spaces and maps at $p$. Looping the diagram in the statement of Lemma~\ref{G+autev} 
gives a homotopy fibration diagram 
\[
\xymatrix{ 
 \Omega {\rm Map}_\nu^\ast(B, BSO(2n))  \ar[d]^{\Omega\chi^\prime_\ast}  \ar[r]^<<<<<{\Omega g_\nu} & \mathcal{G}^{+}(\nu) \ar[r]^{{\Omega {\rm ev}}} \ar[d]^{\chi_\nu} & SO(2n) \ar[d]^{\chi^\prime} \\
 \Omega {\rm Map}_\nu^\ast(B, B{\rm aut}(S^{2n-1})) \ar[r]^<<<{\Omega g_s} & {\rm aut}_B(\partial V)\ar[r]^{{{\rm ev}}} & {\rm aut}(S^{2n-1}).   
   } 
\] 
Since $n\geq 2$ we have $\pi_{1}(SO(n))\cong\mathbb{Z}/2$. Therefore, as $p\in\mathcal{Q}_{k+2,l+2}$, 
we have $p\neq 2$, so the composite 
$S^1\stackrel{\tau}{\longrightarrow} \mathcal{G}^{+}(\nu)\stackrel{\Omega {\rm ev}}{\longrightarrow} SO(2n)$ 
is null homotopic. Thus $\tau$ lifts to a map $\widetilde{\tau}: S^1\longrightarrow \Omega {\rm Map}_{\nu}^\ast(B, BSO(2n))$ such that $\tau\simeq \Omega g_\nu\circ \widetilde{\tau}$.
The homotopy commutativity of this homotopy fibration diagram therefore implies that $\chi_\nu\circ\tau$ is homotopic to the composite
\[
S^1\stackrel{\widetilde{\tau}}\longrightarrow \Omega {\rm Map}_{\nu}^\ast(B, BSO(2n)) \stackrel{\Omega \chi'_\ast}{\longrightarrow} \Omega {\rm Map}_{\nu}^\ast(B, B{\rm aut}(S^{2n-1})) \stackrel{\Omega g_s}{\longrightarrow} {\rm aut}_B(\partial V).
\] 
We will show that $\Omega \chi'_{\ast}\circ\widetilde{\tau}$ is null homotopic, implying that 
$\Omega g_{s}\circ \Omega \chi'_{\ast}\circ\widetilde{\tau}$ is null homotopic, and therefore  
that $\chi_{\nu}\circ\tau$ is null homotopic, as required. 

It remains to show that $\Omega \chi'_{\ast}\circ\widetilde{\tau}$ is null homotopic, for which we pass to stable homotopy. Consider the homotopy commutative diagram
\begin{equation}\label{taustableeq}
\begin{gathered}
\xymatrix{ 
S^1 \ar[r]^<<<{\widetilde{\tau}}  \ar[dr]_{\widetilde{\tau}_\infty}& 
\Omega {\rm Map}_{\nu}^\ast(B, BSO(2n))  \ar[r]^{\Omega \chi'_\ast}   \ar[d]^{j_{\infty\ast}} &
\Omega {\rm Map}_{\nu}^\ast(B, B{\rm aut}(S^{2n-1}))  \ar[d]^{j_{\infty\ast}} \\
&\Omega {\rm Map}_{\nu}^\ast(B, BSO)  \ar[r]^{\Omega \chi'_\ast}  &
\Omega {\rm Map}_{\nu}^\ast(B, B{\rm aut}(S^{\infty})),
   }  
   \end{gathered}
\end{equation}
where by abuse of notation $SO(2n) \stackrel{j_{\infty}}{\longrightarrow}SO$ and ${\rm aut}(S^{2n-1}) \stackrel{j_{\infty}}{\longrightarrow} {\rm aut}(S^{\infty})$ are the stabilization maps, $\widetilde{\tau}_\infty=j_{\infty\ast}\circ \widetilde{\tau}$, and $SO\stackrel{\chi^\prime}{\longrightarrow} {\rm aut}(S^{\infty})$ is the homotopy colimit of the maps $\chi^\prime$. 
In general, let ${\rm aut}_{1}(S^{m})$ be the component of ${\rm aut}(S^{m})$ containing 
the identity map. There is an evaluation fibration 
\(\nameddright{{\rm aut}_{1}(S^{m})}{}{{\rm aut}(S^{m})}{ev}{S^{m}}\), 
and if ${\rm Map}_{1}(S^{m},S^{m})$ is the component of ${\rm Map}(S^{m},S^{m})$ containing the identity map, 
then the subspace inclusion 
\(\namedright{{\rm aut}_{1}(S^{m})}{}{{\rm Map}_{1}(S^{m},S^{m})}\) 
is a homotopy equivalence. Thus the evaluation fibration induces exact sequences 
\(\nameddright{\pi_{k+m}(S^{m})}{}{\pi_{k}({\rm aut}(S^{m}))}{}{\pi_{k}(S^{m})}\) 
for every $k\geq 1$. In particular, ${\rm aut}(S^{m})$ has the same connectivity as $S^{m}$. 
Consequently, the map $B{\rm aut}(S^{2n-1}) \stackrel{Bj_{\infty}}{\longrightarrow}B {\rm aut}(S^{\infty})$ is a homotopy equivalence up to dimension $2n-2$. By hypothesis, ${\rm dim}(B)+2=l+2\leq 2n-2$, so it follows that the composition $\Omega \chi'_{\ast}\circ\widetilde{\tau}$ is null homotopic if and only if $j_{\infty\ast}\circ \Omega \chi'_{\ast}\circ\widetilde{\tau}$ is. By the homotopy commutativity of (\ref{taustableeq}), it remains to show that $\Omega \chi'_{\ast}\circ\widetilde{\tau}_\infty$ is null homotopic.

Generically, if ${\rm Map}^{\ast}_{f}(X, \Omega Y)$ is the component of the space ${\rm Map}^\ast(X, \Omega Y)$ of pointed 
continuous maps from $X$ to a loop space $\Omega Y$ containing a given map $f$, then as ${\rm Map}^\ast(X,\Omega Y)\cong \Omega  {\rm Map}^\ast(X,Y)$ is a loop space, by \cite[Exercise 3 in Section 3.C]{Hat} there is a natural 
homotopy equivalence ${\rm Map}^{\ast}_{f}(X, \Omega Y)\simeq {\rm Map}_{0}^\ast( X,\Omega Y)$, 
where the right side is the component of ${\rm Map}^\ast(X,\Omega Y)$ containing the basepoint.
It follows that $\Omega {\rm Map}^{\ast}_{f}(X, \Omega Y) \simeq \Omega {\rm Map}_{0}^\ast( X, \Omega Y)= \Omega {\rm Map}^\ast(X,\Omega Y)\cong {\rm Map}^\ast(X,\Omega^2 Y)$. Hence, as it is well known that both $BSO$ and $B{\rm aut}(S^{\infty})$ are infinite loop spaces, the composition $\Omega \chi'_\ast\circ \widetilde{\tau}_\infty$ in Diagram (\ref{taustableeq}) can be rewritten as
\[
S^1\stackrel{\widetilde{\tau}_\infty}{\longrightarrow}  {\rm Map}^\ast(B, SO) \stackrel{\chi'_\ast}{\longrightarrow}  {\rm Map}^\ast(B, {\rm aut}(S^{\infty})).
\]

It remains to show that $\chi'_{\ast}\circ\widetilde{\tau}_\infty$ is null homotopic.
The adjoint of 
$\chi'_{\ast}\circ\widetilde{\tau}_\infty$ is the composite 
\[ 
\Sigma B\stackrel{\widetilde{\tau}^\#_\infty}{\longrightarrow} SO \stackrel{\chi^\prime}{\longrightarrow} {\rm aut}(S^{\infty}), 
\] 
where $\widetilde{\tau}^\#_\infty$ is the adjoint of $\widetilde{\tau}_\infty$. Consider the longer composite
\[
\Sigma B\stackrel{\widetilde{\tau}^\#_\infty}{\longrightarrow} SO\stackrel{\chi^\prime}{\longrightarrow} {\rm aut}(S^{\infty})\stackrel{j}{\hookrightarrow}{\rm Map}(S^{\infty}, S^{\infty})
\] 
where $j$ is the inclusion of the components corresponding to maps of degree $\pm 1$. The composite $SO\stackrel{\chi^\prime}{\longrightarrow}{\rm aut}(S^{\infty})  \stackrel{j}{\hookrightarrow} {\rm Map}(S^{\infty}, S^{\infty})$ is the standard stabilized action $\chi$. Hence, by Lemma \ref{QimJ=0lemma} the composite 
$j\circ\chi'\circ\widetilde{\tau}^\#_\infty$ is null homotopic. Since $\Sigma B$ is path-connected, 
its image in ${\rm aut}(S^{\infty})$ lies in a single path-component, for which $j$ has a left inverse. Thus  
$\chi^\prime\circ \widetilde{\tau}^\#_\infty$ is null homotopic. Hence its adjoint $\chi'_\ast\circ \widetilde{\tau}_\infty$ 
is also null homotopic, as required. 
\end{proof}

\noindent 
\textbf{Proof of Theorems \ref{loopblowupthmintro}, \ref{loopblowupthmintrorational} and \ref{dichotomythmintro}}.
We first prove the homotopy decomposition for the loop space of the blow-up $\widetilde{N}$ 
in Theorem \ref{loopblowupthmintro}. Recall from~(\ref{S1tildeEeq}) that there is a principal $S^{1}$-bundle 
\(\nameddright{S^{1}}{\widetilde{j}}{\widetilde{E}}{\widetilde{s}}{\widetilde{N}}\). 
We will establish a loop space decomposition of $\widetilde{E}$ using the methods in \cite[Section 2]{HT2}. 
For the remainder of this subsection, suppose that $l\leq 2n-4$ and localize 
spaces and maps at a prime $p$ satisfying the conditions of Theorem \ref{loopblowupthmintro}. By Lemma \ref{Bslemma}, the prime $p$ satisfies $p>\frac{1}{2}(l-k)+1$, $p \not\in \mathcal{Q}_{k+2,l+2}$ and $H_\ast(B;\mathbb{Z})$ is $p$-torsion free. In particular, we can apply Proposition~\ref{hextlemma}. 

Recall from Proposition~\ref{tildeEprop} that there is a geometric pushout 
\[\begin{gathered}
\xymatrix{
\partial V\times S^1 \ar[r]^{{\rm id} \times \iota} \ar[d]^{(\iota_c\times {\rm id})\circ \phi} & 
\partial V \times D^2 \ar[d]\\
N_c\times S^1 \ar[r] &
\widetilde{E}. 
}
\end{gathered}\]

\begin{lemma} 
   \label{CQ} 
  There exists a map 
   \(t\colon\namedright{\widetilde{E}}{}{N_c}\) 
   such that the composite 
   \(\nameddright{N_c\times S^1}{}{\widetilde{E}}{t}{N_c}\) 
   is homotopic to the projection onto the first factor and the composite 
   \(\nameddright{\partial V\times D^2}{}{\widetilde{E}}{t}{N_c}\) 
   is homotopic to 
   \(\namedright{\partial V\times D^2}{s}{N_c}\). 
\end{lemma} 

\begin{proof}
Consider the diagram 
\[\label{tdiag}
\begin{gathered}
\xymatrix{ 
  \partial V\times S^{1}\ar[r]^{{\rm id}\times \iota}\ar[d]^{(\iota_c\times {\rm id})\circ \phi} & \partial V\times D^2\ar[d]\ar@/^/[ddr]^{s} &  \\ 
   N_c\times S^{1}\ar[r]\ar@/_/[drr]_{\pi_{1}} & \widetilde{E} \ar@{.>}[dr]^(0.4){t} & \\ 
   & & N_c\, , 
   }  
   \end{gathered}
\]
where the inner square is a geometric pushout by~(\ref{tildeEpushoutdiag}) and the outer square homotopy 
commutes by~(\ref{sdiag}). Since a geometric pushout is a homotopy pushout, there is a map $t$ that makes the two triangular regions homotopy commute.  
\end{proof}

Consequently, if $i_1: N_c\longrightarrow N_c\times S^1$ is the inclusion of the first factor 
then the composite 
\(\namedddright{N_c}{i_{1}}{N_c\times S^{1}}{}{\widetilde{E}}{t}{N_c}\) 
is homotopic to $\pi_{1}\circ i_{1}$, which is homotopic to the identity map on~$N_c$. 
This proves the following.
\begin{lemma} 
   \label{tinverse} 
      The map $t: \widetilde{E}\longrightarrow N_c$ has a right homotopy inverse.~$\qqed$
\end{lemma}

We now identify some homotopy fibres associated to the maps in in Lemma~\ref{CQ}. 
Define the spaces~$F$ and $H$ and the map $g$ by the homotopy fibrations  
\begin{equation} 
\label{N0cube1} 
\begin{split} 
\nameddright{F}{g}{\partial V}{\iota_c}{N_{c}} \\ 
\nameddright{H}{}{\widetilde{E}}{t}{N_{c}}.  
\end{split} 
\end{equation} 
If 
\(i_{1}\colon\namedright{\partial V}{}{\partial V\times D^2}\) 
is the inclusion of the first factor, then as $s\circ i_{1}\simeq \iota_c$ there is a homotopy commutative diagram of 
homotopy fibrations 
\[\diagram 
     F\rto^-{g}\dto^{f'} & \partial V\rto^-{\iota_c}\dto^{i_{1}} & N_{c}\ddouble \\ 
     F'\rto & \partial V\times D^2\rto^-{s} & N_{c}.  
  \enddiagram\] 
that defines the space $F'$ and the map $f'$.   
Since $i_{1}$ is a homotopy equivalence, the five-lemma applied to the long exact sequence of 
homotopy groups implies that $f'$ induces an isomorphism on homotopy groups and so is a 
homotopy equivalence since all spaces have the homotopy type of $CW$-complexes. 
Thus there is a homotopy fibration 
\begin{equation} 
\label{N0cube2} 
\nameddright{F}{g'}{\partial V\times D^2}{s}{N_{c}} 
\end{equation} 
where $g'=i_{1}\circ g$. Next, define the space $F''$ and the map $f''$ by the homotopy commutative diagram of 
homotopy fibrations  
\[\diagram 
       F''\rto\dto^{f''} & \partial V\times S^1\rto^-{s\circ(id\times\iota)}\dto^{id\times\iota} 
           & N_{c}\ddouble \\ 
       F\rto^-{g'} & \partial V\times D^2\rto^-{s} & N_{c}. 
  \enddiagram\] 
This homotopy fibration diagram implies that $F''$ is the homotopy pullback of $g'$ and $id\times\iota$. 
Since $g'= i_{1}\circ g$, there is an iterated homotopy pullback diagram 
\begin{equation} 
  \label{N0cubea} 
  \diagram 
      F\times S^1\rto^-{g\times id}\dto^{\pi_{1}} & \partial V\times S^1\rto^-{id\times id}\dto^{\pi_{1}} 
           & \partial V\times S^1\dto^{id\times\iota} \\ 
      F\rto^-{g} & \partial V\rto^-{i_{1}} & \partial V\times D^2. 
  \enddiagram 
\end{equation}  
Here, the right square is a homotopy pullback since $D^2$ is contractible, and the left square 
is a homotopy pullback by the naturality of the projection $\pi_{1}$. Since the outer rectangle 
is the homotopy pullback of $g^\prime= i_{1}\circ g$ and $id\times\iota$, we see that $F''\simeq F\times S^1$ 
and $f''\simeq\pi_{1}$. Thus there is a homotopy fibration 
\begin{equation} 
\label{N0cube3} 
\llnameddright{F\times S^1}{g\times id}{\partial V\times S^1}{s\circ(id\times\iota)}{N_{c}}. 
\end{equation} 

Therefore, composing each of the four corners of~(\ref{tildeEpushoutdiag}) with the map 
\(\namedright{\widetilde{E}}{t}{N_c}\) 
and taking homotopy fibres, from Lemma \ref{CQ}, (\ref{N0cube1}), (\ref{N0cube2}) and~(\ref{N0cube3}) 
we obtain a homotopy commutative cube 
\begin{equation}
  \label{Qcube} 
  \spreaddiagramcolumns{0.2pc}\spreaddiagramrows{-1pc} 
   \diagram
      F\times S^{1}\rrto^-{a}\drto^-{b}\ddto^-(0.33){g\times id} & & F\dline^-{}\drto & \\
      & S^{1}\rrto\ddto^(0.25){i_{2}} & \dto & H\ddto \\
      \partial V\times S^{1}\rline^(0.6){{\rm id}\times \iota}\drto_(0.4){(\iota_c\times {\rm id})\circ \phi\ \ } & \rto & \partial V\times D^2\drto & \\
      & N_c\times S^{1}\rrto & & \widetilde{E}, 
  \enddiagram 
\end{equation}
in which the bottom face is a homotopy pushout and the four sides are homotopy 
pullbacks, the map~$i_2$ is the inclusion to the second factor, and the maps $a$ and $b$ are induced 
maps of fibres. Mather's Cube Lemma~(Theorem \ref{2cubthm}) implies that the top face is a homotopy pushout. 

We would like to identify the maps $a$ and $b$ in~(\ref{Qcube}). 
As notation, let $I=[0,1]$ be the unit interval with basepoint at $0$. For 
pointed spaces $X$ and $Y$ the (reduced) \emph{join} $X\ast Y$ is defined as the quotient space 
\[X\ast Y=(X\times I\times Y)/\sim\] 
where $(x,0,y)\sim (x',0,y)$, $(x,1,y)\sim (x,1,y')$ and $(\ast,t,\ast)\sim (\ast,0,\ast)$ for all 
$x,x'\in X$, $y,y'\in Y$ and $t\in I$. Equivalently, if $CX$ and $CY$ are the reduced cones on $X$ 
and $Y$ respectively, then there is a pushout 
\[\diagram 
       X\times Y\rto\dto & X\times CY\dto \\ 
       CX\times Y\rto & X\ast Y. 
  \enddiagram\] 
As $CX$ and $CY$ are contractible, this implies that there is a homotopy pushout 
\begin{equation} 
  \label{joinpo} 
  \diagram 
        X\times Y\rto^-{\pi_{1}}\dto^{\pi_{2}} & X\dto \\ 
        Y\rto & X\ast Y 
  \enddiagram 
\end{equation} 
where $\pi_{1}$ and $\pi_{2}$ are the projections. It is well known that there is a homotopy equivalence 
$X\ast Y\simeq\Sigma X\wedge Y$. 

\begin{lemma} 
   \label{cubeab} 
   The maps $a$ and $b$ in~(\ref{Qcube}) are homotopic to the projections 
   onto the first and second factor respectively. Consequently, there is a homotopy equivalence 
   $H\simeq F\ast S^{1}\simeq \Sigma^{2}F$. 
\end{lemma} 

\begin{proof} 
The rear face of the cube~(\ref{Qcube}) is the iterated homotopy pullback in~(\ref{N0cubea}), 
therefore $a\simeq\pi_{1}$. The map $b$ is induced by the homotopy pullback diagram 
\[\diagram 
      F\times S^{1}\rto^-{g\times {\rm id}}\dto^{b} & \partial V\times S^{1}\rto^-{s\circ ({\rm id}\times \iota)}\dto^{(\iota_c\times {\rm id})\circ \phi} & N_c\ddouble \\  
      S^{1}\rto^-{i_{2}} & N_c\times S^{1}\rto^-{\pi_{1}} & N_c,
  \enddiagram\] 
where the right square commutes up to homotopy by Lemma \ref{CQ}. Generically, let $\pi_{2}$ 
be the projection onto the second factor. Observe that by definition of $\phi$, we have $(\pi_2\circ (\iota_c\times {\rm id})\circ \phi)(v, t)=\pi_2(\iota_c(\tau(t)v), t)=t=\pi_2(v, t)$, that is, $\pi_2\circ (\iota_c\times {\rm id})\circ \phi=\pi_2$. Thus $\pi_2\circ (\iota_c\times {\rm id})\circ \phi\circ (g\times {\rm id})=\pi_{2}\circ(g\times id)=\pi_2$, and therefore $b=\pi_2\circ i_2\circ b$ is homotopic to $\pi_2$. 
Therefore the homotopy pushout giving $H$ in the top face of~(\ref{Qcube}) 
is equivalent, up to homotopy, to that given by the projections 
\(\namedright{F\times S^{1}}{\pi_{1}}{F}\) 
and 
\(\namedright{F\times S^{1}}{\pi_{2}}{S^{1}}\). 
From the discussion before the lemma, this homotopy pushout gives the join $F\ast S^{1}$. 
\end{proof} 

We are now ready to prove Theorem \ref{loopblowupthmintro}.

\begin{proof}[Proof of Theorem \ref{loopblowupthmintro}]
Recall the principal $S^1$-bundle 
\(\nameddright{S^{1}}{}{\widetilde{E}}{}{\widetilde{N}}\)
from~(\ref{S1tildeEeq}) and recall that $F$ is the homotopy fibre of the inclusion $\iota_c:\partial V\hookrightarrow N_c$ of the boundary. 
For the homotopy fibration
\[
H\stackrel{}{\longrightarrow} \widetilde{E}\stackrel{t}{\longrightarrow} N_c
\]
in (\ref{N0cube1}), Lemma~\ref{cubeab} implies that $H\simeq \Sigma^2 F$. 
Since $t$ has a right homotopy inverse by Lemma \ref{tinverse}, this homotopy fibration splits after looping to give a homotopy equivalence
\[\label{tildeEspliteq}
\Omega \widetilde{E}\simeq \Omega N_c \times \Omega \Sigma^2 F.
\]
Combining this homotopy equivalence with Lemma \ref{canbundletriv}, we obtain the 
asserted homotopy equivalence $\Omega \widetilde{N}\simeq S^1\times \Omega N_c\times \Omega \Sigma^2 F$.
\end{proof} 

Theorem~\ref{loopblowupthmintrorational} can be proved as a corollary of Theorem~\ref{loopblowupthmintro}. 

\begin{proof}[Proof of Theorem~\ref{loopblowupthmintrorational}] 
Since $B$ is a finite dimensional manifold, the homology $H_\ast(B;\mathbb{Z})$ is $p$-torsion free for sufficient large primes $p$. Similarly, for sufficient large primes $p$, $p>\frac{1}{2}(l-k)+1$ and $(p-1) \nmid 2s$ for any $k+2\leq 4s \leq l+2$. Then we can choose a large enough prime $p$ satisfying the three conditions on $p$ in Theorem \ref{loopblowupthmintro}, and hence Theorem \ref{loopblowupthmintro} can apply to give a $p$-local homotopy equivalence
\[
\Omega\widetilde{N}\simeq S^{1}\times\Omega N_{c}\times\Omega\Sigma^{2} F.
\] 
A further localization away from $p$ now implies the asserted rational homotopy equivalence.
\end{proof}

Using Theorem \ref{loopblowupthmintrorational} we can now also prove Theorem \ref{dichotomythmintro}. 
Write $A\simeq_{\mathbb{Q}} B$ for a rational homotopy equivalence between spaces $A$ and $B$.

\begin{proof}[Proof of Theorem \ref{dichotomythmintro}]
By Theorem \ref{loopblowupthmintrorational} there is a rational homotopy equivalence 
\[\Omega \widetilde{N}\simeq_{\mathbb{Q}} S^1\times \Omega N_c\times\Omega \Sigma^2 F,\] 
implying that 
\[
\pi_\ast(\widetilde{N})\otimes \mathbb{Q}\cong \mathbb{Q}(2)\oplus (\pi_\ast(N_c)\otimes \mathbb{Q})\oplus (\pi_\ast(\Sigma^2 F)\otimes \mathbb{Q}),
\] 
where $\mathbb{Q}(i)$ a copy of the vector space $\mathbb{Q}$ of degree $i$.
Recall that $F$ is the homotopy fibre of the inclusion $\iota_c: \partial V\stackrel{}{\rightarrow} N_c$. Since $N_c$ is a smooth compact $(l+2n)$-manifold, we have \mbox{$H^{l+2n}(N_c,\partial V;\mathbb{Z})\cong \mathbb{Z}$}. This implies that $\iota_c$ is not a rational homotopy equivalence. In particular $F$ is not rationally contractible. 
Hence $\Sigma^2 F$ is rationally a wedge of spheres, which is rationally hyperbolic unless $F\simeq_{\mathbb{Q}} S^q$ for some $q$.
It follows that $\widetilde{N}$ is rationally elliptic if and only if $N_c$ is rationally elliptic and $F\simeq_{\mathbb{Q}} S^q$ for some $q$. In this case 
\[
\pi_\ast(\widetilde{N})\otimes \mathbb{Q}\cong \mathbb{Q}(2) \oplus (\pi_\ast(N_c)\otimes \mathbb{Q})\oplus (\pi_\ast(S^{q+2})\otimes \mathbb{Q}),
\] 
and $\pi_\ast(S^{q+2})\otimes \mathbb{Q}\cong \mathbb{Q}(q+2)$ if $q$ is odd or 
$\pi_\ast(S^{q+2})\otimes \mathbb{Q}\cong\mathbb{Q}(q+2)\oplus \mathbb{Q}(2q+3)$ if $q$ is even.
Otherwise $\widetilde{N}$ is rationally hyperbolic, and this happens if and only if either $N_c$ is rationally hyperbolic or $F\not\simeq_{\mathbb{Q}} S^q$ for any $q$. This proves the theorem.
\end{proof}

\part{Blow ups from a homotopy theoretic point of view}
\label{part2}
In this part, we study blow ups from a purely homotopy theoretic point of view.
In Section \ref{sec:topblowup}, we show that there is a homotopy decomposition of $\Omega\widetilde{N}$ given a condition concerning a natural free circle action on $\partial V$ (\ref{SPnudiag}).
In particular, we prove Theorem \ref{loopNtildeintro} as Theorem \ref{loopNtilde}.
The free circle action is a special case of the homotopy action associated to a principal homotopy fibration.
In Section \ref{sec: act}, we analyze such homotopy actions in general using methods from~\cite{BT2}, and relate it to Whitehead products. This allows us to determine when the condition for Theorem \ref{loopNtildeintro} holds for blow ups at a point in Section \ref{sec: stab}.
In \cite{HT2}, we studied the homotopy theory of the connected sums $N\#\mathbb{C}P^n$ and $N\#\mathbb{H}P^{\frac{n}{2}}$, which can be viewed as the complex and quaternionic blow up of $N$ at a point. 
Here, by different methods, we recover the result of \cite{HT2} in the complex case, and refine it in the quaternionic case. In particular, we prove Theorem \ref{stabthmintro}.

\section{The integral homotopy theory of a blow up} 
\label{sec:topblowup} 
In Part \ref{part1}, to study the loop homotopy of the geometric blow up $\widetilde{N}$ \eqref{blowuppushoutdiag2}, we studied the canonical circle bundle $S^{1}\longrightarrow\widetilde{E}\longrightarrow\widetilde{N}$ in~(\ref{S1tildeEeq}). In Proposition \ref{tildeEprop}, we gave a geometric description of $\widetilde{E}$ that was crucial to proving the local homotopy decomposition of $\Omega\widetilde{N}$ in Theorem \ref{loopblowupthmintro}. In this section, from a more homotopy theoretic point of view, we study the topological blow up $\widetilde{N}$ (\ref{blowuppushoutdiag}) through a principal homotopy fibration $S^{1}\longrightarrow\widetilde{E}\longrightarrow\widetilde{N}$, and prove an integral decomposition of $\Omega\widetilde{N}$ under a condition depending on the existence of a certain homotopy. Further analysis and applications will be given in the subsequent sections.

We follow the notation from the previous sections. Consider the topological blow up $\widetilde{N}$ defined by the pushout (\ref{blowuppushoutdiag}):
\[ 
\xymatrix{
\partial V \ar[r]^{q} \ar[d]^{\iota_c} &
P\nu\ar[d]^{j} \\
N_c\ar[r]^{r} &
\widetilde{N}. 
}\]
By Remark \ref{top-geo-blow-rmk}, the topological blow up $\widetilde{N}$ is canonically homotopy equivalent to the geometric blow up in \eqref{blowuppushoutdiag2}. In particular, the canonical circle bundle (\ref{S1tildeEeq}) over the geometric blow up implies a principal homotopy fibration 
\[
S^{1}\longrightarrow\widetilde{E}\longrightarrow\widetilde{N}
\]
over the topological blow up $\widetilde{N}$, and all the homotopy results in the previous sections, including  Lemmas \ref{tildeelemma} and \ref{canbundletriv}, hold for the topological blow up $\widetilde{N}$.

Recall by Lemma~\ref{canbundletriv} there is a homotopy equivalence $\Omega\widetilde{N}\simeq S^{1}\times\Omega\widetilde{E}$. Therefore to identify the homotopy type of $\Omega \widetilde{N}$ it suffices to determine the homotopy type of $\Omega\widetilde{E}$. We will apply Mather's Cube 
Lemma~(Theorem \ref{2cubthm}) twice, first to describe $\widetilde{E}$ as a homotopy pushout and 
then, after showing that there is a map 
\(\namedright{\widetilde{E}}{}{N_{c}}\), 
to describe its homotopy fibre.

Abusing notation, let $\widetilde{e}: \widetilde{N}\stackrel{}{\longrightarrow} \mathbb{C}P^{\infty}$ 
be the map representing $\widetilde{e}\in H^2(\widetilde{N})$ in Lemma \ref{tildeelemma}. 
By that lemma, $N_c\stackrel{r}{\longrightarrow}\widetilde{N}\stackrel{\widetilde{e}}{\longrightarrow} \mathbb{C}P^{\infty}$ is null homotopic while $P\nu\stackrel{j}{\longrightarrow} \widetilde{N}\stackrel{\widetilde{e}}{\longrightarrow} \mathbb{C}P^{\infty}$ is homotopic to the map~$e$ representing the generator $e\in H^2(P\nu)$.
Map all four corners of the pushout to $\widetilde{N}$, compose with $\widetilde{e}$, and take homotopy fibres. 
This results in a homotopy commutative cube  
\begin{equation} 
  \label{NEtilde} 
  \spreaddiagramcolumns{-0.8pc}\spreaddiagramrows{-0.8pc} 
  \diagram
      \partial V\times S^{1}\rrto^-{a}\drto^{b}\ddto^{\pi_1} & & \partial V\dline^<<<<{q}\drto & \\
      & N_{c}\times S^{1}\rrto\ddto^<<<<{\pi_1} & \dto & \widetilde{E}\ddto \\
      \partial V\rline^>>>>{q}\drto^{\iota_c} & \rto & P\nu\drto^{j} & \\
      & N_{c}\rrto^{r} & & \widetilde{N} 
  \enddiagram 
\end{equation} 
for some maps $a$ and $b$, where the four sides are homotopy pullbacks and the bottom face is a 
homotopy pushout. By Mather's Cube Lemma~(Theorem \ref{2cubthm}) the top face is also a homotopy pushout. To go further it is necessary 
to identify $a$ and $b$. 

In general, a homotopy fibration 
\[\nameddright{F}{}{E}{}{B}\] 
has a connecting map 
\(\partial\colon\namedright{\Omega B}{}{F}\)  
and a homotopy action 
\[\theta\colon\namedright{F\times\Omega B}{}{F}\] 
whose restriction to $F$ is the identity map and whose restriction to $\Omega B$ is $\partial$. 
The homotopy action has the property that there is a homotopy commutative diagram 
\begin{equation} 
  \label{actiondgrm} 
  \diagram 
      F\times\Omega B\rto^-{\theta}\dto^{\pi_{1}} & F\dto \\ 
      F\rto & E, 
  \enddiagram 
\end{equation} 
where $\pi_1$ is the projection onto the first factor.  
 
\begin{lemma} 
   \label{abidentify} 
   The map $a$ in~(\ref{NEtilde}) is homotopic to the homotopy action 
   \(\theta\colon\namedright{\partial V\times S^{1}}{}{\partial V}\) 
   arising from the homotopy fibration 
   \(\nameddright{\partial V}{q}{P\nu}{e}{\mathbb{C}P^{\infty}}\). 
   The map $b$ in~(\ref{NEtilde}) is homotopic to the product map 
   \(\namedright{\partial V\times S^{1}}{\iota_c\times {\rm id}}{N_{c}\times S^{1}}\). 
\end{lemma} 

\begin{proof} 
Since $\theta$ is a homotopy action, by~(\ref{actiondgrm}) there is a homotopy pullback 
\[
\begin{gathered}
\xymatrix{
\partial V\times S^1 \ar[r]^{\theta} \ar[d]^{\pi_1}  &
\partial V \ar[d]^{q}  \\
\partial V \ar[r]^{q}  &
P\nu.
}
\end{gathered}
\]
This recovers the homotopy pullback in the rear face of (\ref{NEtilde}). Hence $a\simeq \theta$. 
Next, the product of the pullbacks 
\[\diagram
      A\rto^-{=}\dto^{f} & A\dto^{f} & & C\rto\dto^{=} & \ast\dto \\ 
      B\rto^-{=} & B & & C\rto & \ast 
  \enddiagram\] 
is the pullback 
\[\diagram 
      A\times C\rto^-{\pi_{1}}\dto^{f\times {\rm id}} & A\dto^{f} \\ 
      B\times C\rto^-{\pi_{1}} & B. 
  \enddiagram\] 
Applying this to the homotopy pullback in the left face of~(\ref{NEtilde}) shows that $b\simeq \iota_c\times {\rm id}$. 
\end{proof} 

Next, we show that, under a certain hypothesis, there is an alternative description 
of $\widetilde{E}$ as a homotopy pullback, and we then use this to show that $N_{c}$ retracts off 
$\widetilde{E}$ in a strong way.  

\begin{lemma} 
   \label{2tildeElemma} 
   Suppose that there is a homotopy commutative diagram 
   \begin{equation} 
      \label{partialVaction} 
      \diagram 
      \partial V\times S^{1}\rto^-{\theta}\dto^{\pi_{1}} & \partial V\dto^{\iota_c} \\ 
      \partial V\rto^-{\iota_c} & N_{c}.
   \enddiagram 
   \end{equation} 
Then there is a homotopy pushout
   \[
      \label{2tildeEdiag} 
   \xymatrix{ 
      \partial V\times S^{1}\ar[r]^-{\pi_1}\ar[d]^{\iota_c\times {\rm id}} &  \partial V\ar[d]    \\ 
      N_c\times S^{1}\ar[r]  &  \widetilde{E}.    }
         \]
\end{lemma} 
\begin{proof} 
In general, a homotopy action 
\(\namedright{F\times\Omega B}{\theta}{F}\) 
has an associativity property: the composite 
\(\lnameddright{F\times\Omega B\times\Omega B}{\theta\times {\rm id}}{F\times\Omega B}{\theta}{F}\) 
is homotopic to 
\(\lnameddright{F\times\Omega B\times\Omega B}{{\rm id}\times\mu}{F\times\Omega B}{\theta}{F}\), 
where $\mu$ is the loop multiplication on $B$. So if $s,s'$ are the shearing maps 
\[s,  s'\colon\namedright{F\times\Omega B}{}{F\times\Omega B}\] 
defined by $s(x,y)=(\theta(x,y),y)$ and $s'(x,y)=(\theta(x,y^{-1}),y)$ then $s\circ s'$ and $s'\circ s$ 
are homotopic to the identity map on $F\times\Omega B$. In particular, $s$ is a homotopy 
equivalence. Observe as well that the composite  
\(\nameddright{F\times\Omega B}{s}{F\times\Omega B}{\pi_{1}}{F}\) 
is homotopic to $\theta$ while the composite 
\(\nameddright{F\times\Omega B}{s'}{F\times\Omega B}{\pi_{1}}{F}\) 
is homotopic to $\theta\circ({\rm id}\times {\rm -id})$, where id is the identity map on $\Omega B$.  

In our case, consider the diagram 
\begin{equation} 
  \label{sheardgrm} 
  \diagram 
     \partial V\times S^{1}\rto^-{s'}\dto^{\iota_{c}\times {\rm id}} 
          & \partial V\times S^{1}\rto^-{\theta}\dto^{\iota_{c}\times {\rm id}} & \partial V\dto \\ 
     N_{c}\times S^{1}\rdouble & N_{c}\times S^{1}\rto & \widetilde{E}. 
  \enddiagram 
\end{equation} 
We first show that the diagram homotopy commutes. For the left square, 
$(\iota_{c}\times {\rm id})\circ s'$ is determined by its projections to $N_{c}$ 
and $S^{1}$. The projection to~$N_{c}$ gives a string of identities  
\[\pi_{1}\circ(\iota_{c}\times {\rm id})\circ s'=\iota_{c}\circ\pi_{1}\circ s'= 
      \iota_{c}\circ\theta\circ({\rm id}\times {\rm -id})\simeq\iota_{c}\circ\pi_{1}\circ({\rm id}\times {\rm -id})= 
      \iota_{c}\circ\pi_{1}=\pi_{1}\circ(\iota_{c}\times {\rm id}).\] 
Here, from left to right, the first equality holds by the naturality of the projection map, the second holds 
by the property of $s'$ mentioned above, the third holds by the hypothesis that 
$\iota_c\circ\theta\simeq \iota_c\circ\pi_{1}$, and the fourth and fifth hold by the naturality 
of the projection. Thus the left square homotopy commutes after projection to $N_{c}$. 
It also homotopy commutes after projection to $S^{1}$ more straightforwardly, since $s'$ 
projects to the identity map on $S^{1}$. Hence the left square in~(\ref{sheardgrm}) homotopy commutes. 
The right square in~(\ref{sheardgrm}) homotopy commutes since it is the homotopy pushout appearing in 
the top face of the cube~(\ref{NEtilde}) with the maps $a$ and $b$ identified by 
Lemma~\ref{abidentify}. 

Next, since $s'$ is a homotopy equivalence and the right square is a homotopy pushout, 
the homotopy commutativity of~(\ref{sheardgrm}) implies that the outer rectangle is also a homotopy pushout. 
Finally, since $\theta\simeq\pi_{1}\circ s$ and $s\circ s'$ is homotopic to the identity map on 
$\partial V\times S^{1}$, the top row in~(\ref{sheardgrm}) is homotopic to $\pi_{1}$. 
Thus we obtain the homotopy pushout asserted by the lemma. 
\end{proof}

\begin{lemma} 
   \label{ENc} 
  Under the hypothesis of Lemma \ref{2tildeElemma}, there is a map 
   \(\namedright{\widetilde{E}}{}{N_{c}}\) 
   satisfying a homotopy commutative diagram 
   \[\xymatrix{ 
      \partial V\times S^{1}\ar[r]^-{\pi_1}\ar[d]^{\iota_c\times {\rm id}} &  \partial V\ar[d]^{j} \ar@/^/[ddr]^{\iota_c} &  \\ 
      N_c\times S^{1}\ar[r] \ar@/_/[drr]_{\pi_{1}} &  \widetilde{E} \ar[dr] & \\ 
      & & N_{c}. 
   }\]
\end{lemma} 

\begin{proof} 
The inner square is a homotopy pushout by Lemma \ref{2tildeElemma}. The naturality of the projection map implies that the composite 
\(\nameddright{\partial V\times S^{1}}{\iota_c\times {\rm id}}{N_{c}\times S^{1}}{\pi_{1}}{N_{c}}\) 
equals $\iota_c\circ\pi_{1}$. 
The homotopy pushout property of $\widetilde{E}$ therefore implies that there is a map 
\(\namedright{\widetilde{E}}{}{N_{c}}\) 
that makes the entire diagram homotopy commute. 
\end{proof}

Suppose that there is a homotopy commutative diagram~(\ref{partialVaction}) 
so that Lemma~\ref{ENc} holds. Recall that the homotopy fibre of the map 
\(\namedright{\partial V}{\iota_c}{N_{c}}\) 
is defined as $F$. Start with the homotopy pushout in the statement of Lemma~\ref{ENc}. Map all 
four corners of the pushout to~$\widetilde{E}$, compose with the map 
\(\namedright{\widetilde{E}}{}{N_{c}}\), 
and then take homotopy fibres. This results in a homotopy commutative cube 
\[
  \label{EFtilde} 
  \spreaddiagramcolumns{-0.8pc}\spreaddiagramrows{-0.8pc} 
  \diagram
      F\times S^{1}\rrto^-{\pi_1}\drto^{\pi_2}\ddto & & F\dline\drto & \\
      & S^{1}\rrto\ddto & \dto & \widetilde{F}\ddto \\
      \partial V\times S^{1}\rline^-{\pi_1}\drto^{\iota_c\times {\rm id}} & \rto & \partial V\drto & \\
      & N_{c}\times S^{1}\rrto & & \widetilde{E}, 
  \enddiagram 
\]
where all four sides are homotopy pullbacks, and $\pi_1$ and $\pi_2$ are the projections. 
By Mather's Cube Lemma~(Theorem \ref{2cubthm}), the top face of the cube is also a homotopy pushout. 
The top face is an instance of the homotopy pushout~(\ref{joinpo}), so we 
immediately obtain the following.

\begin{lemma} 
   \label{tildeFidentify} 
   There is a homotopy equivalence $\widetilde{F}\simeq S^{1}\ast F$.  ~$\qqed$
\end{lemma} 

This leads to the following. 

\begin{proposition} 
   \label{loopEtilde} 
   Suppose that there is a homotopy commutative diagram~(\ref{partialVaction}). Then 
   there is a homotopy fibration 
   \(\nameddright{\Sigma^{2} F}{}{\widetilde{E}}{}{N_{c}}\) 
   that splits after looping to give a homotopy equivalence 
   \[\Omega\widetilde{E}\simeq\Omega N_{c}\times\Omega\Sigma^{2} F.\] 
\end{proposition} 

\begin{proof} 
By definition, $\widetilde{F}$ is the homotopy fibre of the map 
\(\namedright{\widetilde{E}}{}{N_{c}}\) 
and by Lemma~\ref{tildeFidentify} there is a homotopy equivalence 
$\widetilde{F}\simeq S^{1}\ast F\simeq\Sigma^{2} F$. 
By Lemma~\ref{ENc}, the composite 
\(\nameddright{N_{c}\times S^{1}}{}{\widetilde{E}}{}{N_{c}}\) 
is homotopic to the projection onto the first factor. In particular, the right map has a 
right homotopy inverse. This implies that, after looping in order to multiply, the 
homotopy fibration 
\(\nameddright{\Omega\Sigma^{2} F}{}{\Omega\widetilde{E}}{}{\Omega N_{c}}\) 
splits as 
$\Omega\widetilde{E}\simeq\Omega N_{c}\times\Omega\Sigma^{2} F$. 
\end{proof} 

Combining Lemma~\ref{canbundletriv} with Proposition~\ref{loopEtilde} we obtain a 
decomposition of the loops on the topological blow up. Theorem~\ref{loopNtilde} restates 
Theorem~\ref{loopNtildeintro}.

\begin{theorem} 
   \label{loopNtilde} 
   Suppose that there is a homotopy commutative diagram~(\ref{partialVaction}). Then 
   there is a homotopy equivalence 
   \[\Omega\widetilde{N}\simeq S^{1}\times\Omega N_{c}\times\Omega\Sigma^{2} F\] 
   where $F$ is the homotopy fibre of the inclusion 
   \(\namedright{\partial V}{}{N_{c}}\).~$\qqed$ 
\end{theorem} 

It remains to determine conditions for when~(\ref{partialVaction}) homotopy commutes. For this purpose, 
it is crucial to understand the homotopy action $\theta$. In the next section this action 
will be related to Whitehead products in the context of a principal fibration. For now, 
we end this section with an example when~(\ref{partialVaction}) homotopy commutes 
for other reasons.
 
As in Example~\ref{tdoubex}, let $N$ be a closed manifold that can be decomposed as a union
\[
N\cong D\nu\cup_g D\nu,
\] 
where $\nu$ is an oriented real vector bundle over a manifold $B$ and $g$ is a self diffeomorphism 
of the boundary $S\nu\cong \partial V$. The manifold $N$ is the {\it twisted double} of $D\nu\cong V$. 
\begin{corollary}\label{doublecoro}
For the topological blow up $\widetilde{N}$ of the twisted double $N$ of $V$, 
there is a homotopy equivalence
\[\Omega\widetilde{N}\simeq S^{1}\times\Omega B\times\Omega S^{2n+1}.\] 
\end{corollary}
\begin{proof}
For convenience, write $N\cong V_1\cup_g V_2$ with $V_1=V_2=V$. Suppose that the topological blow up $\widetilde{N}$ is constructed from the embedding $V_1{}{\hookrightarrow} N$. Then $N_c=V_2$ and the boundary inclusion $\partial V\stackrel{\iota_c}{\hookrightarrow} N_c$ becomes $\partial V_1=\partial V_2 \stackrel{\iota_c}{\hookrightarrow}V_2$ with homotopy fibre $F\simeq S^{2n-1}$.

Consider the diagram
\begin{equation}\label{twisteddoubledgrm} 
\begin{gathered}
\xymatrix{
\partial V_1\times S^1 \ar[r]^{\theta} \ar[d]^{\pi_1}  &
\partial V_1 \ar[d]^{q}   \ar[r]^{s_\nu}   &
B \ar@{=}[d]  \ar[r]^{s_0} &
V_2  \ar@{=}[d]  \\
\partial V_1 \ar[r]^{q}  &
P\nu  \ar[r]^{p_\nu}  &
B  \ar[r]^{s_0} &
V_2.  
}
\end{gathered}
\end{equation} 
The left square is a homotopy pullback by~(\ref{actiondgrm}), the middle square homotopy commutes by (\ref{SPnudiag}), and~$s_0$ is the zero section of $V_2$. For either copy of $V$, the projection 
\(\namedright{V}{p_{\nu}}{B}\) 
is a homotopy equivalence with the inclusion 
\(\namedright{B}{s_{0}}{V}\) 
as its homotopy inverse. Therefore, as $s_{\nu}$ is the composite 
\(\nameddright{\partial V_{1}}{\iota_c}{V_{1}}{p_{\nu}}{B}\), 
the composite $s_{0}\circ s_{\nu}$ is homotopic to $\iota_c$. The homotopy commutativity 
of~(\ref{twisteddoubledgrm}) implies that $s_{0}\circ s_{\nu}\simeq s_{0}\circ p_{\nu}\circ q$. 
Thus the outer perimeter of~(\ref{twisteddoubledgrm}) gives $\iota_c\circ\theta\simeq \iota_c\circ\pi_{1}$, 
implying that~(\ref{partialVaction}) homotopy commutes. Therefore by Theorem \ref{loopNtildeintro} there is a homotopy equivalence
 \[\Omega\widetilde{N}\simeq S^{1}\times\Omega V_2\times\Omega\Sigma^{2} S^{2n-1}\simeq S^{1}\times\Omega B\times\Omega S^{2n+1}.\] 
\end{proof}


\section{The homotopy action of a principal homotopy fibration}
\label{sec: act}
Suppose that $G$ is an $H$-group, that is, $G$ is a homotopy associative $H$-space 
with a homotopy inverse. The commutator on $G$ is the map 
\(c\colon\namedright{G\times G}{}{G}\) 
defined pointwise by 
$c(g,h)=ghg^{-1}h^{-1}$. Observe that $c$ is null homotopic when restricted to $G\vee G$, 
resulting in a quotient map 
\[\overline{c}\colon\namedright{G\wedge G}{}{G}.\] 
The homotopy class of $\overline{c}$ is uniquely determined by that of $c$ since the 
connecting map for the homotopy cofibration 
\(\nameddright{G\vee G}{}{G\times G}{}{G\wedge G}\) 
is null homotopic, implying that for any space~$Z$ there is an injection 
\(\namedright{[G\wedge G,Z]}{}{[G\times G,Z]}\). 
Given maps 
\(\bar{f}\colon\namedright{X}{}{G}\) 
and 
\(\bar{g}\colon\namedright{Y}{}{G}\), 
the \emph{Samelson product} of $\bar{f}$ and $\bar{g}$ is the composite 
\[\langle\bar{f},\bar{g}\rangle\colon\nameddright{X\wedge Y}{\bar{f}\wedge\bar{g}}{G\wedge G}{\overline{c}}{G}.\] 

Suppose that $G=\Omega Z$ and there are maps 
\(f\colon\namedright{\Sigma X}{}{Z}\) 
and 
\(g\colon\namedright{\Sigma Y}{}{Z}\). 
Let 
\(\bar{f}\colon\namedright{X}{}{\Omega Z}\) 
and 
\(\bar{g}\colon\namedright{Y}{}{\Omega Z}\) 
be the adjoints of $f$ and $g$ respectively. The \emph{Whitehead product} 
\[[f,g]\colon\namedright{\Sigma X\wedge Y}{}{Z}\] 
is the adjoint of the Samelson product $\langle\bar{f},\bar{g}\rangle$. 

Now suppose that there is a principal fibration 
\(\nameddright{\Omega Z}{\partial}{E}{p}{B}\) 
induced by a map 
\(\namedright{B}{\varphi}{Z}\). 
Let 
\(\theta\colon\namedright{\Omega Z\times E}{}{E}\) 
be the canonical homotopy action associated to the principal fibration and let $\vartheta$ 
be the composite 
\[\vartheta\colon\llnameddright{\Omega B\times E}{\Omega\varphi\times {\rm id}}{\Omega Z\times E}{\theta}{E}.\] 
Since the restriction of $\theta$ to $\Omega Z$ is $\partial$, the restriction of $\vartheta$ to $\Omega B$ 
is $\partial\circ\Omega\varphi$, which is null homotopic. Therefore $\vartheta$ factors as 
\(\nameddright{\Omega B\times E}{}{\Omega B\ltimes E}{\overline{\vartheta}}{E}\) 
for some map $\overline{\vartheta}$. There may be different choices of~$\overline{\vartheta}$. 
In~\cite[Proposition 2.9, Corollary 2.12]{BT2}, recorded here as Theorem~\ref{BTWhitehead}, 
it is shown that a choice can be made that behaves well with respect to Whitehead products. 

As general notation, for spaces $A$ and $B$, the \emph{right half-smash} of $A$ and $B$ is the quotient space 
\[A\rtimes B=(A\times B)/(\ast\times B).\] 
Let 
\(ev\colon\namedright{\Sigma\Omega X}{}{X}\) 
be the canonical evaluation map. If $A$ is a suspension then it is well known that there is a 
natural homotopy equivalence 
\[e\colon\namedright{(\Sigma B\wedge A)\vee\Sigma A}{}{B\ltimes\Sigma A}.\] 

\begin{theorem} 
   \label{BTWhitehead}  
   Let 
   \(\nameddright{\Omega Z}{}{E}{p}{B}\) 
   be a principal fibration induced by a map 
   \(\namedright{B}{\varphi}{Z}\).  
   Let 
   \(f\colon\namedright{\Sigma X}{}{B}\) 
   and 
   \(g\colon\namedright{\Sigma Y}{}{E}\) 
   be maps. Then there is a choice of $\overline{\vartheta}$ with the property that there 
   is a homotopy commutative diagram 
   \[\diagram 
         \Omega\Sigma X\ltimes\Sigma Y\rto^-{\Omega f\ltimes g}\dto^{e^{-1}} 
             &  \Omega B\ltimes E\rto^-{\overline{\vartheta}} & E\dto^{p} \\ 
         (\Sigma\Omega\Sigma X\wedge Y)\vee\Sigma Y\rrto^-{[f\circ ev,p\circ g]+p\circ g} & & B. 
     \enddiagram\] 
\end{theorem} 
\vspace{-1cm}~$\qqed$\bigskip 

A refined version of Theorem~\ref{BTWhitehead} appears in~\cite[Theorem 3.4]{BT1} 
under the additional assumption that $\Omega\varphi$ has a right homotopy inverse. The hypotheses 
in that paper are not quite relevant to our situation so we state and prove the version that is needed here. 

\begin{proposition} 
   \label{HTWhitehead} 
   Let 
   \(\nameddright{\Omega Z}{}{E}{p}{B}\) 
   be a principal fibration induced by a map 
   \(\namedright{B}{\varphi}{Z}\) 
   and suppose that $\Omega\varphi$ has a right homotopy inverse 
   \(s\colon\namedright{\Omega Z}{}{\Omega B}\). 
   Let 
   \(g\colon\namedright{\Sigma Y}{}{E}\) 
   be a map. Then there is a homotopy commutative diagram 
   \[\diagram 
        \Omega Z\ltimes\Sigma Y\rto^-{s\ltimes g}\dto^{e^{-1}} & \Omega B\ltimes E\rto^-{\overline{\vartheta}} 
             & E\dto^{p} \\ 
        (\Sigma\Omega Z\wedge Y)\vee\Sigma Y\rrto^-{[\gamma,p\circ g]+p\circ g} & & B  
     \enddiagram\] 
   where $\gamma$ is the composite 
   \(\nameddright{\Sigma\Omega Z}{\Sigma s}{\Sigma\Omega B}{ev}{B}\). 
\end{proposition}    

\begin{proof} 
In general, let 
\(E\colon\namedright{X}{}{\Omega\Sigma X}\) 
be the suspension map that is adjoint to the identity map on~$\Sigma X$. Consider the diagram 
\[\diagram 
     \Omega Z\ltimes\Sigma Y\rrto^-{E\ltimes 1}\dto^{e^{-1}} 
          & & \Omega\Sigma\Omega Z\ltimes\Sigma Y\rto^-{\Omega\gamma\ltimes g}\dto^{e^{-1}} 
          & \Omega B\ltimes E\rto^-{\overline{\vartheta}} & E\dto^{p} \\ 
     (\Sigma\Omega Z\wedge Y)\vee\Sigma Y\rrto^-{(\Sigma E\wedge 1)\vee 1}
          & & (\Sigma\Omega\Sigma\Omega Z\wedge Y)\vee\Sigma Y\rrto^-{[\gamma\circ ev,p\circ g]+p\circ g} 
          & & B. 
  \enddiagram\] 
The left square homotopy commutes by the naturality of the homotopy equivalence $e$. 
The right rectangle homotopy commutes by Theorem~\ref{BTWhitehead} with $f=\gamma$. 
Along the bottom row, the naturality of the Whitehead product implies that 
$[\gamma\circ ev,p\circ g]\circ(\Sigma E\wedge 1)\simeq [\gamma\circ ev\circ\Sigma E,p\circ g]$. 
Since $ev$ is a left homotopy inverse for $\Sigma E$, we have $\gamma\circ ev\circ\Sigma E\simeq\gamma$. 
Thus the bottom row of the diagram is homotopic to $[\gamma,p\circ g]+p\circ g$. Therefore 
the homotopy commutativity of the diagram implies that 
$p\circ\overline{\vartheta}\circ(\Omega\gamma\ltimes g)\circ(E\ltimes 1)\simeq ([\gamma,p\circ g]+p\circ g)\circ e^{-1}$. 

Next consider the diagram   
\[\diagram 
     \Omega Z\ltimes\Sigma Y\rto^-{E\ltimes 1}\dto^{s\ltimes 1} 
          & \Omega\Sigma\Omega Z\ltimes\Sigma Y\rto^-{\Omega\gamma\ltimes g}\dto^{\Omega\Sigma s\ltimes 1} 
          & \Omega B\ltimes E \\ 
     \Omega B\ltimes\Sigma Y\rto^-{E\ltimes 1} 
          & \Omega\Sigma\Omega B\ltimes\Sigma Y\urto_-{\Omega ev\ltimes g} & 
  \enddiagram\] 
The left square homotopy commutes by the naturality of $E$ and the right triangle commutes by 
definition of $\gamma$. Since $\Omega ev$ is a left homotopy inverse for $E$, the homotopy 
commutativity of the diagram implies that 
$(\Omega\gamma\ltimes g)\circ(E\ltimes 1)\simeq(s\ltimes g)$. Thus we obtain 
$p\circ\overline{\vartheta}\circ(\Omega\gamma\ltimes g)\circ(E\ltimes 1)\simeq 
     p\circ\overline{\vartheta}\circ(s\ltimes g)$. 
     
The two homotopies for $p\circ\overline{\vartheta}\circ(\Omega\gamma\ltimes g)\circ(E\ltimes 1)$ 
therefore give $p\circ\overline{\vartheta}\circ(s\ltimes g)\simeq ([\gamma,p\circ g]+p\circ g)\circ e^{-1}$, 
as asserted. 
\end{proof} 

There is a finer articulation of Proposition~\ref{HTWhitehead}. Since $\Omega\varphi$ has a 
right homotopy inverse, the fibration connecting map 
\(\namedright{\Omega Z}{}{E}\) 
is null homotopic, so the homotopy action 
\(\namedright{\Omega Z\times E}{\theta}{E}\) 
factors through a map 
\(\namedright{\Omega Z\ltimes E}{}{E}\). 
There is a preferred choice for this factorization. Consider the diagram 
\[\diagram 
      \Omega Z\times E\rto^-{s\ltimes 1}\dto & \Omega B\times E\rto^-{\Omega\varphi\times {\rm id}}\dto 
          & \Omega Z\times E\rto^-{\theta} & E \\ 
      \Omega Z\ltimes E\rto^-{s\ltimes 1} & \Omega B\ltimes E.\urrto_-{\overline{\vartheta}} & & 
  \enddiagram\] 
The right square commutes by the naturality of the quotient map to the half-smash and the right 
triangle homotopy commutes by definition of $\overline{\vartheta}$. Since $s$ is a right homotopy 
inverse for $\Omega\varphi$, the top row is homotopic to $\theta$. Thus if $\overline{\theta}$ is defined 
as the composite 
\[\overline{\theta}\colon\nameddright{\Omega Z\ltimes E}{s\ltimes 1}{\Omega B\ltimes E} 
      {\overline{\vartheta}}{E}\] 
then there is a homotopy commutative diagram 
\[\diagram 
      \Omega Z\times E\rto^-{\theta}\dto & E \\ 
      \Omega Z\ltimes E.\urto_-{\overline{\theta}} & 
  \enddiagram\] 
Moreover, substituting the definition of $\overline{\theta}$ into Proposition~\ref{HTWhitehead} 
immediately gives the following reformulation in terms of the action of $\Omega Z$ on $E$. 

\begin{proposition} 
   \label{HTWhitehead2} 
   Let 
   \(\nameddright{\Omega Z}{}{E}{p}{B}\) 
   be a principal fibration induced by a map 
   \(\namedright{B}{\varphi}{Z}\) 
   and suppose that $\Omega\varphi$ has a right homotopy inverse 
   \(s\colon\namedright{\Omega Z}{}{\Omega B}\). 
   Let 
   \(g\colon\namedright{\Sigma Y}{}{E}\) 
   be a map. Then there is a homotopy commutative diagram 
   \[\diagram 
        \Omega Z\ltimes\Sigma Y\rto^{1\ltimes g}\dto^{e^{-1}} 
             & \Omega Z\ltimes E\rto^-{\overline{\theta}} & E\dto^{p} \\ 
        (\Sigma\Omega Z\wedge Y)\vee\Sigma Y\rrto^-{[\gamma,p\circ g]+p\circ g} & & B  
     \enddiagram\] 
   where $\gamma$ is the composite 
   \(\nameddright{\Sigma\Omega Z}{\Sigma s}{\Sigma\Omega B}{ev}{B}\).~$\qqed$  
\end{proposition}   

The special case when $E$ is a suspension is worth singling out. 

\begin{corollary} 
   \label{HTWhiteheadcor} 
   Let 
   \(\nameddright{\Omega Z}{}{\Sigma E}{p}{B}\) 
   be a principal fibration induced by a map 
   \(\namedright{B}{\varphi}{Z}\) 
   and suppose that $\Omega\varphi$ has a right homotopy inverse 
   \(s\colon\namedright{\Omega Z}{}{\Omega B}\). 
   Then there is a homotopy commutative diagram 
   \[\diagram 
        \Omega Z\ltimes\Sigma E\rto^{\overline{\theta}}\dto^{e^{-1}} & \Sigma E\dto^{p} \\ 
        (\Sigma\Omega Z\wedge E)\vee\Sigma E\rto^-{[\gamma,p]+p} & B  
     \enddiagram\] 
   where $\gamma$ is the composite 
   \(\nameddright{\Sigma\Omega Z}{\Sigma s}{\Sigma\Omega B}{ev}{B}\).  
\end{corollary}  

\begin{proof} 
Take $\Sigma Y=\Sigma E$ and $g$ equal to the identity map in Proposition~\ref{HTWhitehead2}. 
\end{proof} 

These results are now applied to analyze the homotopy action
\(\namedright{\partial V\times S^{1}}{\theta}{\partial V}\) 
that appeared in Diagram (\ref{partialVaction}) of Theorem~\ref{loopNtilde} for topological blow ups. Consider 
the principal fibration 
\[\nameddright{S^{1}}{\delta}{\partial V}{q}{P\nu}\] 
induced by the map 
\(\namedright{P\nu}{e}{\mathbb{C}P^{\infty}}\). 

\begin{lemma} 
   \label{deltatriv} 
   The map 
   \(\namedright{\Omega P\nu}{\Omega e}{S^{1}}\) 
   has a right homotopy inverse. 
\end{lemma} 

\begin{proof} 
By Lemma~\ref{esphere}, there is a map 
\(\namedright{S^{2}}{}{P\nu}\) 
with the property that the composite 
\(\nameddright{S^{2}}{}{P\nu}{e}{\mathbb{C}P^{\infty}}\) 
induces an isomorphism in degree $2$ cohomology. If 
\(\namedright{S^{1}}{E}{\Omega S^{2}}\) 
is the suspension (equivalent to the inclusion of the bottom cell), then the composite 
\(\namedddright{S^{1}}{E}{\Omega S^{2}}{}{\Omega P\nu}{\Omega e}{S^{1}}\) 
induces an isomorphism in cohomology and so is a homotopy equivalence. 
Thus $\Omega e$ has a right homotopy inverse.  
\end{proof} 

Lemma~\ref{deltatriv} lets us apply Proposition~\ref{HTWhitehead2} to 
immediately obtain the following.

\begin{lemma} 
   \label{Pnuexamplelem} 
   Consider the principal fibration  
   \(\nameddright{S^{1}}{}{\partial V}{q}{P\nu}\) 
   induced by the map 
   \(\namedright{P\nu}{e}{\mathbb{C}P^{\infty}}\). 
   For any map 
   \(g\colon\namedright{\Sigma Y}{}{\partial V}\) 
   there is a homotopy commutative diagram 
   \[\diagram 
        S^{1}\ltimes\Sigma Y\rto^{1\ltimes g}\dto^{e^{-1}} 
             &S^{1}\ltimes\partial V\rto^-{\overline{\theta}} & \partial V\dto^{p} \\ 
        (\Sigma S^{1}\wedge Y)\vee\Sigma Y\rrto^-{[\gamma,p\circ g]+p\circ g} & & P\nu  
     \enddiagram\]  
   where $\gamma$ is the composite 
   \(\nameddright{\Sigma S^{1}}{\Sigma s}{\Sigma\Omega P\nu}{ev}{P\nu}\).~$\qqed$  
\end{lemma} 

\begin{remark} 
\label{Pnuexampleremark}
The map $\gamma$ in Lemma~\ref{Pnuexamplelem}, or its ocurrance in Corollary~\ref{HTWhiteheadcor} 
when $\Sigma E=\partial V$, can sometimes be described more simply. 
Consider the diagram 
\[\diagram 
     \Sigma S^{1}\rto^-{\Sigma s}\drdouble & \Sigma\Omega P\nu\rto^-{ev}\dto^{\Sigma\Omega e} & P\nu\dto^{e} \\ 
     & \Sigma\Omega\mathbb{C}P^{\infty}\rto^-{ev} & \mathbb{C}P^{\infty}. 
  \enddiagram\] 
The left triangle homotopy commutes since $s$ is a right homotopy inverse for $\Omega e$ and 
the right square homotopy commutes by the naturality of the evaluation map. The lower direction 
around the diagram is therefore homotopic to the inclusion of the bottom cell. The top row is the 
definition of $\gamma$, so $\gamma$ is a lift of the inclusion 
\(\namedright{S^{2}}{}{\mathbb{C}P^{\infty}}\). 
Consequently, for the homotopy fibration 
\(\nameddright{\partial V}{}{P\nu}{e}{\mathbb{C}P^{\infty}}\), 
if $\partial V$ is $2$-connnected then $e$ induces an isomorphism on $\pi_{2}$, 
implying that $\gamma$ is homotopic to the inclusion of the bottom cell into $P\nu$. 
\end{remark}

The special case of a blow up at a point will be discussed in more detail in the next section.


\section{The homotopy of a stabilized manifold revisited}
\label{sec: stab}
In \cite{HT2} we studied the homotopy theory of a stabilized manifold such as $N \# \mathbb{C}P^n$ or $N\# \mathbb{H}P^{\frac{n}{2}}$, which can be viewed as the complex or quaternionic blow up of $N$ at a point. 
One may refer to \cite{GGS} for more details on the quaternionic blow up at a point.
In this section, we will apply the results in Section \ref{sec: act} to recover the results of \cite{HT2} by a purely homotopy theoretic method, with the bonus of a slight improvement.

Let us consider the complex case first. Recall as in Example~\ref{blowuppointex}, the geometric and topological blow ups at a point coincide, and for $n\geq 2$ the defining pushout of the blow up at a point $\widetilde{N}:=N \# \mathbb{C}P^n$ reduces to a pushout 
\[\diagram 
     S^{2n-1}\rto^-{q_{0}}\dto^{i_{c}} & \mathbb{C}P^{n-1}\dto \\ 
     N_{0}\rto & \widetilde{N}, 
  \enddiagram\] 
  where $N_0$ is the manifold $N$ with one open disk removed and $q_0$ is the standard projection.
In this case the principal fibration 
\(\nameddright{S^{1}}{}{\partial V}{q}{P\nu}\)
of the blow up $\widetilde{N}$ is 
\(\nameddright{S^{1}}{}{S^{2n-1}}{q_0}{\mathbb{C}P^{n-1}}\). 
Lemma~\ref{Pnuexamplelem} therefore implies that there is a homotopy commutative diagram 
\begin{equation}
\label{*actwheq}
\diagram 
     S^{1}\ltimes S^{2n-1}\rrto^{\overline{\theta}}\dto^{e^{-1}} & & S^{2n-1}\dto^{q_0} \\ 
     (S^{1}\wedge S^{2n-1})\vee S^{2n-1}\rrto^-{[\epsilon,q_0]+q_0} & & \mathbb{C}P^{n-1}  
  \enddiagram
  \end{equation}
where $\epsilon$ is the inclusion of the bottom cell by Remark~\ref{Pnuexampleremark}.

In~\cite{BJS} the Whitehead product $[\epsilon,q_0]$ was described explicitly. If $n$ is even 
then $[\epsilon,q_0]$ is null homotopic. If $n$ is odd then $[\epsilon,q_0]$ is nontrivial and 
homotopic to the composite 
\(\nameddright{S^{2n}}{\eta}{S^{2n-1}}{q_0}{\mathbb{C}P^{n-1}}\), 
where $\eta$ represents a generator of $\pi_{2n}(S^{2n-1})\cong\mathbb{Z}/2\mathbb{Z}$ for $n\geq 2$. 
Since 
\(\namedright{S^{2n-1}}{q_{0}}{\mathbb{C}P^{n-1}}\) 
induces an isomorphism on $\pi_{2n}$ and $\pi_{2n-1}$, the diagram~(\ref{*actwheq}) implies the following. 

\begin{lemma} 
   \label{thetabarfactor} 
   There is a homotopy commutative diagram 
   \[\diagram 
         S^{1}\ltimes S^{2n-1}\rto^-{\overline{\theta}}\dto^{e^{-1}} & S^{2n-1}\ddouble \\ 
         S^{2n}\vee S^{2n-1}\rto^-{t\cdot\eta+1} & S^{2n-1} 
      \enddiagram\] 
    where $t\in\mathbb{Z}/2\mathbb{Z}$ equals $1$ if $n$ is odd and $0$ if $n$ is even.~$\qqed$  
\end{lemma}  

This leads to a condition for when~(\ref{partialVaction}) holds in the case of a blow up at a point. 

\begin{proposition} 
   \label{} 
   If $n$ is even then there is a homotopy commutative diagram 
   \[\diagram 
        S^{1}\times S^{2n-2}\rto^-{\theta}\dto^{\pi_{2}} & S^{2n-1}\dto^{i_{c}} \\ 
        S^{2n-1}\rto^-{i_{c}} & N_{0}. 
     \enddiagram\] 
   If $n$ is odd then this diagram homotopy commutes after localization away from $2$. 
\end{proposition} 

\begin{proof} 
Consider the diagram 
\[\diagram 
       S^{1}\times S^{2n-1}\rto^-{\theta}\dto & S^{2n-1}\ddouble \\ 
       S^{1}\ltimes S^{2n-1}\rto^-{\overline{\theta}}\dto^{\pi} & S^{2n-1}\ddouble \\ 
       S^{2n-1}\rto^-{=} &  S^{2n-1}  
  \enddiagram\] 
where $\pi$ is the projection. The upper square homotopy commutes by definition of $\overline{\theta}$. 
If $n$ is even then the lower square homotopy commutes by Lemma~\ref{thetabarfactor}. The composite 
along the left column is the projection $\pi_{2}$. Thus $\theta\simeq\pi_{2}$. Now compose with $i_{c}$  
to get $i_{c}\circ\theta\simeq i_{c}\circ\pi_{2}$, as asserted. 

If $n$ is odd then the lower square in the diagram above homotopy commutes after localization 
away from $2$ by Lemma~\ref{thetabarfactor} since the map $\eta$ localizes trivially. Now proceed 
as in the even case.  
\end{proof}

Thus by Theorem \ref{loopNtilde} we recover the following result from~\cite{HT2} that was obtained using a more geometric argument. 

\begin{corollary} 
   \label{BJScor} 
   In the case of a complex blow up at a point and $n\geq 2$, if $n$ is even then there is a homotopy equivalence 
   \[\Omega(N\conn\mathbb{C}P^{n})\simeq S^{1}\times\Omega N_{0}\times\Omega\Sigma^{2} F\] 
   where $F$ is the homotopy fibre of the inclusion 
   \(\namedright{S^{2n-1}}{}{N_{0}}\). If $n$ is odd then this homotopy equivalence exists after 
   localizing away from~$2$.~$\qqed$ 
\end{corollary} 

There is an analogous result in the quaternionic case. 
Suppose $n=2s\geq 4$.
Start with the principal fibration 
\(\nameddright{S^{3}}{}{S^{4s-1}}{q_0}{\mathbb{H}P^{s-1}}\) 
induced by the map 
\(\namedright{\mathbb{H}P^{s-1}}{u}{\mathbb{H}P^{\infty}}\)
corresponding to a generator $u\in H^4(\mathbb{H}P^{s-1})$. 
Arguing as above, there is a homotopy commutative diagram 
\[\diagram 
      S^{3}\times S^{4s-1}\rto^-{\theta}\dto^{\pi_{2}} & S^{4s-1}\dto^{q_0} \\ 
      S^{4s-1}\rto^-{q_0} & \mathbb{H}P^{s-1} 
  \enddiagram\] 
if and only if the Whitehead product 
\(\namedright{\Sigma S^{3}\wedge S^{4s-2}}{[\epsilon,q_0]}{\mathbb{H}P^{s-1}}\) 
is null homotopic, where $\epsilon$ is the inclusion of the bottom cell. In~\cite{BJS} it was shown that $[\epsilon,q_0]$ is homotopic to the composite 
\(\lnameddright{S^{4s+2}}{(\pm s)\cdot\nu}{S^{4s-1}}{q_0}{\mathbb{H}P^{s-1}}\), 
where $\nu$ represents a generator of $\pi_{4s+2}(S^{4s-1})\cong\mathbb{Z}/24\mathbb{Z}$ for $s\geq 2$. 
On the other hand, a quaternionic version of Theorem \ref{loopNtilde} can be proved for $N\conn\mathbb{H}P^{s}$ by the same argument in Section \ref{sec:topblowup}. Thus we have the following, which refines a result in~\cite{HT2} that was stated only after localizing away from $2$ and $3$.  

\begin{corollary} 
   \label{BJScor2} 
   In the case of a quaternionic blow up at a point and $s\geq 2$, if $s\equiv 0\bmod{24}$ then 
   there is a homotopy equivalence 
   \[\Omega(N\conn\mathbb{H}P^{s})\simeq S^{3}\times\Omega N_{0}\times\Omega\Sigma^{4} F\] 
   where $F$ is the homotopy fibre of the inclusion 
   \(\namedright{S^{4s-1}}{}{N_{0}}\). If $s\not\equiv 0\bmod{24}$ then this homotopy equivalence 
   exists after localizing away from the primes dividing~$\frac{24}{(24,s)}$.
   ~$\qqed$ 
\end{corollary} 

%

\bibliographystyle{amsalpha}

\end{sloppypar}
\end{document}